\documentclass[12pt]{amsart}
\usepackage{graphicx}
\usepackage{amsmath}
\usepackage{amssymb}
\usepackage{setspace}
\newtheorem{thm}{Theorem}[section]

\newtheorem{lem}[thm]{Lemma}

\theoremstyle{definition}

\theoremstyle{remark}
\newtheorem{rem}[thm]{Remark}
\theoremstyle{conclusion}

\theoremstyle{question}
\numberwithin{equation}{section}

\begin{document}
\title[Non-critical higher order Lane-Emden-Hardy equations]{Liouville type theorems, a priori estimates and existence of solutions for non-critical higher order Lane-Emden-Hardy equations}

\author{Wei Dai, Shaolong Peng, Guolin Qin}

\address{School of Mathematics and Systems Science, Beihang University (BUAA), Beijing 100083, P. R. China}
\email{weidai@buaa.edu.cn}

\address{School of Mathematics and Systems Science, Beihang University (BUAA), Beijing 100083, P. R. China}
\email{SLPeng@buaa.edu.cn}

\address{Institute of Applied Mathematics, Chinese Academy of Sciences, Beijing 100190, and University of Chinese Academy of Sciences, Beijing 100049, P. R. China}
\email{qinguolin18@mails.ucas.ac.cn}

\thanks{Wei Dai is supported by the NNSF of China (No. 11501021).}

\begin{abstract}
In this paper, we are concerned with the non-critical higher order Lane-Emden-Hardy equations
\begin{equation*}
  (-\Delta)^{m}u(x)=\frac{u^{p}(x)}{|x|^{a}} \,\,\,\,\,\,\,\,\,\,\,\, \text{in} \,\,\, \mathbb{R}^{n}
\end{equation*}
with $n\geq3$, $1\leq m<\frac{n}{2}$, $0\leq a<2m$, $1<p<\frac{n+2m-2a}{n-2m}$ if $0\leq a<2$, and $1<p<+\infty$ if $2\leq a<2m$. We prove Liouville theorems for nonnegative classical solutions to the above Lane-Emden-Hardy equations (Theorem \ref{Thm0}), that is, the unique nonnegative solution is $u\equiv0$. As an application, we derive a priori estimates and existence of positive solutions to non-critical higher order Lane-Emden equations in bounded domains (Theorem \ref{Thm1} and \ref{Thm2}). The results for critical order Hardy-H\'{e}non equations have been established by Chen, Dai and Qin \cite{CDQ} recently.
\end{abstract}
\maketitle {\small {\bf Keywords:} Lane-Emden-Hardy equations; Liouville theorems; Nonnegative solutions; Super poly-harmonic properties; Method of moving planes in local way; Blowing-up analysis. \\

}

\section{Introduction}

In this paper, we first investigate the Liouville property of nonnegative solutions to the following non-critical higher order Lane-Emden-Hardy equations
\begin{equation}\label{PDE}\\\begin{cases}
(-\Delta)^{m}u(x)=\frac{u^{p}(x)}{|x|^{a}} \,\,\,\,\,\,\,\,\,\, \text{in} \,\,\, \mathbb{R}^{n}, \\
u(x)\geq0, \,\,\,\,\,\,\,\, x\in\mathbb{R}^{n},
\end{cases}\end{equation}
where $u\in C^{2m}(\mathbb{R}^{n})$ if $-\infty<a\leq0$, $u\in C^{2m}(\mathbb{R}^{n}\setminus\{0\})\cap C^{2m-2}(\mathbb{R}^{n})$ if $0<a<2m$, $n\geq3$, $1<p<\frac{n+2m-2a}{n-2m}$ if $0\leq a<2$, and $1<p<+\infty$ if $2\leq a<2m$.

For $0<\alpha\leq n$, PDEs of the form
\begin{equation}\label{GPDE}
  (-\Delta)^{\frac{\alpha}{2}}u(x)=\frac{u^{p}(x)}{|x|^{a}}
\end{equation}
are called the fractional order or higher order Hardy (Lane-Emden, H\'{e}non) equations for $a>0$ ($a=0$, $a<0$, respectively), which have many important applications in conformal geometry and Sobolev inequalities. We say equations \eqref{GPDE} have critical order if $\alpha=n$ and non-critical order if $0<\alpha<n$. Liouville type theorems for equations \eqref{GPDE} (i.e., nonexistence of nontrivial nonnegative solutions) have been quite extensively studied (see \cite{BG,CD,CDQ,CFY,CL,CL1,CLiu,DQ1,DQ,GS,Lei,Lin,MP,P,PS,WX} and the references therein). It is crucial in establishing a priori estimates and existence of positive solutions for non-variational boundary value problems of a class of elliptic equations (see \cite{BM,CDQ,CL3,CL4,GS1,PQS}).

In the special case $a=0$, equation \eqref{GPDE} becomes the well-known Lane-Emden equation, which also arises as a model in astrophysics. For $\alpha=2$ and $1<p<p_{s}:=\frac{n+2}{n-2}$ ($:=\infty$ if $n=2$), Liouville type theorem was established by Gidas and Spruck in their celebrated article \cite{GS}. Later, the proof was simplified to a large extent by Chen and Li in \cite{CL} using the Kelvin transform and the method of moving planes (see also \cite{CL1}). For $n>\alpha=4$ and $1<p<\frac{n+4}{n-4}$, Lin \cite{Lin} proved the Liouville type theorem for all the nonnegative $C^{4}(\mathbb{R}^{n})$ smooth solutions of \eqref{GPDE}. When $\alpha\in(0,n)$ is an even integer and $1<p<\frac{n+\alpha}{n-\alpha}$, Wei and Xu established Liouville type theorem for all the nonnegative $C^{\alpha}(\mathbb{R}^{n})$ smooth solutions of \eqref{GPDE} in \cite{WX}. For general $a\in\mathbb{R}$, $0<\alpha\leq n$, $0<p<\min\{\frac{n+\alpha-2a}{n-\alpha},\frac{n+\alpha-a}{n-\alpha}\}$ ($1<p<+\infty$ if $\alpha=n$), there are also lots of literatures on Liouville type theorems for general fractional order or higher order Hardy-H\'{e}non equations \eqref{GPDE}, for instance, Bidaut-V\'{e}ron and Giacomini \cite{BG}, Chen, Dai and Qin \cite{CDQ}, Chen and Fang \cite{CF}, Cheng and Liu \cite{CLiu}, Dai and Qin \cite{DQ1}, Gidas and Spruck \cite{GS}, Lei \cite{Lei}, Mitidieri and Pohozaev \cite{MP}, Phan \cite{P}, Phan and Souplet \cite{PS} and many others. For Liouville type theorems on systems of PDEs of type \eqref{GPDE} with respect to various types of solutions (e.g., stable, radial, nonnegative, sign-changing, $\cdots$), please refer to \cite{BG,DQ,FG,M,P,PQS,S,SZ} and the references therein.

For the critical nonlinearity cases $p=\frac{n+\alpha}{n-\alpha}$ with $a=0$ and $0<\alpha<n$, the quantitative and qualitative properties of solutions to fractional order or higher order conformally invariant equations \eqref{GPDE} have also been widely studied. In the special case $n>\alpha=2$, equation \eqref{GPDE} becomes the well-known Yamabe problem (for related results, please see Gidas, Ni and Nirenberg \cite{GNN1,GNN}, Caffarelli, Gidas and Spruck \cite{CGS} and the references therein). For $n>\alpha=4$, Lin \cite{Lin} classified all the positive $C^{4}$ smooth solutions of \eqref{GPDE}. In \cite{WX}, among other things, Wei and Xu proved the classification results for all the positive $C^{\alpha}$ smooth solutions of \eqref{GPDE} when $\alpha\in(0,n)$ is an even integer. For $n>\alpha=3$, Dai and Qin \cite{DQ1} classified the positive $C^{3,\epsilon}_{loc}\cap\mathcal{L}_{1}$ classical solutions of \eqref{GPDE}. In \cite{CLO}, by developing the method of moving planes in integral forms, Chen, Li and Ou classified all the positive $L^{\frac{2n}{n-\alpha}}_{loc}$ solutions to the equivalent integral equation of the PDE \eqref{GPDE} for general $\alpha\in(0,n)$, as a consequence, they obtained the classification results for positive weak solutions to PDE \eqref{GPDE}. Subsequently, Chen, Li and Li \cite{CLL} developed a direct method of moving planes for fractional Laplacians $(-\Delta)^{\frac{\alpha}{2}}$ with $0<\alpha<2$ and classified all the $C^{1,1}_{loc}\cap\mathcal{L}_{\alpha}$ positive solutions to the PDE \eqref{GPDE} directly as an application, where the function space
\begin{equation}\label{0-0}
  \mathcal{L}_{\alpha}(\mathbb{R}^{n}):=\Big\{f: \mathbb{R}^{n}\rightarrow\mathbb{R}\,\big|\,\int_{\mathbb{R}^{n}}\frac{|f(x)|}{1+|x|^{n+\alpha}}dx<\infty\Big\}.
\end{equation}
In the limiting (i.e., critical order) cases $\alpha=n$, there are also a large amount of literatures on classification results for positive solutions to the following critical order conformally invariant equations with exponential nonlinearities
\begin{equation}\label{0-3}
  (-\Delta)^{\frac{n}{2}}u=(n-1)!e^{nu},
\end{equation}
for instance, Chen and Li \cite{CL1}, Chang and Yang \cite{CY}, Chen and Zhang \cite{CZ}, Lin \cite{Lin}, Wei and Xu \cite{WX} and Zhu \cite{Zhu}. For more literatures on the quantitative and qualitative properties of solutions to fractional order or higher order conformally invariant PDE and IE problems, please refer to \cite{CD,CL1,CZ,DFHQW,DFQ,Zhu} and the references therein.

In this paper, we will establish Liouville type theorem for nonnegative classical solutions of \eqref{PDE} in the cases $1<p<\frac{n+2m-2a}{n-2m}$ if $0\leq a<2$ and $1<p<+\infty$ if $2\leq a<2m$. Lei \cite{Lei} has proved the nonexistence of positive solutions to \eqref{PDE} for $0\leq a<2$ and $1<p<\frac{n-a}{n-2m}$. One should note that, our results extend the range $p\in(1,\frac{n-a}{n-2m})$ and $0\leq a<2$ in \cite{Lei} to the full range $1<p<\frac{n+2m-2a}{n-2m}$ if $0\leq a<2$ and $1<p<+\infty$ if $2\leq a<2m$.

Our Liouville type result for \eqref{PDE} is the following theorem.
\begin{thm}\label{Thm0}
Assume $n\geq3$, $1\leq m<\frac{n}{2}$, $0\leq a<2m$, $1<p<\frac{n+2m-2a}{n-2m}$ if $0\leq a<2$, $1<p<+\infty$ if $2\leq a<2m$, and $u$ is a nonnegative solution of \eqref{PDE}. If one of the following two assumptions
\begin{equation*}
  0\leq a\leq2+2p \,\,\,\,\,\,\,\,\,\,\,\, \text{or} \,\,\,\,\,\,\,\,\,\,\,\, u(x)=o(|x|^{2}) \,\,\,\, \text{as} \,\, |x|\rightarrow+\infty
\end{equation*}
holds, then $u\equiv0$ in $\mathbb{R}^{n}$.
\end{thm}

\begin{rem}\label{remark1}
In \cite{CLiu}, Cheng and Liu proved Liouville type theorem for \eqref{PDE} in the cases $a<0$ and $1<p<\frac{n+2m-a}{n-2m}$ (there is actually an extra assumption $p>\frac{n}{n-2m}$ in \cite{CLiu}, but it is clear from their proof that the assumption $p>\frac{n}{n-2m}$ is redundant and unnecessary). Among other things, Lei \cite{Lei} established the nonexistence of positive solutions to \eqref{PDE} for $0\leq a<2$ and $1<p<\frac{n-a}{n-2m}$. However, we found a few technical mistakes in their proof, more precisely, in their proof of super poly-harmonic properties (see Theorem 2 in \cite{CLiu} and Theorem 2.1 in \cite{Lei}). For instance, the possibility that constant $C_{\ast}=0$ have to be ruled out in the proof of Theorem 2 in \cite{CLiu}, and a factor $R^{-a}$ should be added to the last inequality in the proof of Theorem 2.1 in \cite{Lei} since $R$ is sufficiently large (thus the assumption $a<2$ is needed therein). In this paper, we will prove the super poly-harmonic properties in Theorem \ref{lemma0} via a unified approach for both $a<0$ and $a\geq0$, as a consequence, we repair the proof in \cite{CLiu} and extend the results in \cite{Lei}.
\end{rem}

\begin{rem}\label{remark3}
For $0<a<2m$, if we consider the nonnegative solutions $u\in C^{2m}(\mathbb{R}^{n}\setminus\{0\})\cap C(\mathbb{R}^{n})$, then it is clear from our proof of Theorem \ref{Thm0} that Liouville theorem as Theorem \ref{Thm0} also holds for $1<p<\frac{n+2m-2a}{n-2m}$ (see Section 2). The main difference is, instead of Theorem \ref{lemma0}, we will show super poly-harmonic properties except the origin $0\in\mathbb{R}^{n}$, that is, $(-\Delta)^{i}u\geq0$ in $\mathbb{R}^{n}\setminus\{0\}$ for $i=1,\cdots,m-1$ (see remark \ref{remark4}).
\end{rem}

\begin{rem}\label{remark0}
In Theorem \ref{Thm0} and Remark \ref{remark3}, the smoothness assumption on $u$ at $x=0$ is necessary. Equation \eqref{PDE} admits a distributional solution of the form $u(x)=C|x|^{-\sigma}$ with $\sigma=\frac{2m-a}{p-1}>0$.
\end{rem}

We also consider the following higher order Navier problem
\begin{equation}\label{tNavier}\\\begin{cases}
(-\Delta)^{m}u(x)=u^{p}(x)+t \,\,\,\,\,\,\,\,\,\, \text{in} \,\,\, \Omega, \\
u(x)=-\Delta u(x)=\cdots=(-\Delta)^{m-1}u(x)=0 \,\,\,\,\,\,\,\, \text{on} \,\,\, \partial\Omega,
\end{cases}\end{equation}
where $n\geq3$, $1\leq m<\frac{n}{2}$, $\Omega\subset\mathbb{R}^{n}$ is a bounded domain with $C^{2m-2}$ boundary $\partial\Omega$ and $t\geq0$.

Theorem 6 in Chen, Fang and Li \cite{CFL} implies immediately the following a priori estimates for any positive solution $u$ to \eqref{tNavier}.
\begin{thm}\label{Thm-CFL} (\cite{CFL})
Assume $\frac{n}{n-2m}<p<\frac{n+2m}{n-2m}$. Then, for any positive solution $u\in C^{2m}(\Omega)\cap C^{2m-2}(\overline{\Omega})$ to the higher order Navier problem \eqref{tNavier}, we have
\begin{equation*}
  \|u\|_{L^{\infty}(\overline{\Omega})}\leq C(n,m,p,\Omega).
\end{equation*}
\end{thm}

As an application of the Liouville theorems (Theorem \ref{Thm0}), we can prove the following a priori estimates for any positive solution $u$ to \eqref{tNavier} via the method of moving planes in local way and blowing-up methods (for related literatures on these methods, please see \cite{BC,BM,CDQ,CL3,CL4,CY0,Li,SZ1}). Our a priori estimates extend the range of $p$ in Theorem \ref{Thm-CFL} remarkably.
\begin{thm}\label{Thm1}
Assume $1<p<\frac{n+2m}{n-2m}$. If one of the following two assumptions
\begin{equation*}
  \text{i)} \,\,\, \Omega \,\, \text{is strictly convex}, \,\, 1<p<\frac{n+2m}{n-2m}, \quad\quad\quad\, \text{or} \quad\quad\quad\, \text{ii)} \,\,\, 1<p\leq\frac{n+2}{n-2}
\end{equation*}
holds. Then, for any positive solution $u\in C^{2m}(\Omega)\cap C^{2m-2}(\overline{\Omega})$ to the higher order Navier problem \eqref{tNavier}, we have
\begin{equation*}
  \|u\|_{L^{\infty}(\overline{\Omega})}\leq C(n,m,p,t,\lambda_{1},\Omega),
\end{equation*}
where $\lambda_{1}$ is the first eigenvalue for $(-\Delta)^{m}$ in $\Omega$ with Navier boundary conditions.
\end{thm}

As a consequence of the a priori estimates (Theorem \ref{Thm-CFL} and Theorem \ref{Thm1}), by applying the Leray-Schauder fixed point theorem, we can derive the following existence result for positive solution to the following Navier problem for higher order Lane-Emden equations
\begin{equation}\label{Navier}\\\begin{cases}
(-\Delta)^{m}u(x)=u^{p}(x) \,\,\,\,\,\,\,\,\,\, \text{in} \,\,\, \Omega, \\
u(x)=-\Delta u(x)=\cdots=(-\Delta)^{m-1}u(x)=0 \,\,\,\,\,\,\,\, \text{on} \,\,\, \partial\Omega,
\end{cases}\end{equation}
where $n\geq3$, $1\leq m<\frac{n}{2}$ and $\Omega\subset\mathbb{R}^{n}$ is a bounded domain with $C^{2m-2}$ boundary $\partial\Omega$.
\begin{thm}\label{Thm2}
Assume $1<p<\frac{n+2m}{n-2m}$. If one of the following two assumptions
\begin{equation*}
  \text{i)} \,\, \Omega \,\, \text{is strictly convex}, \,\, 1<p<\frac{n+2m}{n-2m}, \quad\, \text{or} \quad\, \text{ii)} \,\, p\in\left(1,\frac{n+2}{n-2}\right]\bigcup\left(\frac{n}{n-2m},\frac{n+2m}{n-2m}\right)
\end{equation*}
holds. Then, the higher order Navier problem \eqref{Navier} possesses at least one positive solution $u\in C^{2m}(\Omega)\cap C^{2m-2}(\overline{\Omega})$. Moreover, the positive solution $u$ satisfies
\begin{equation*}
  \|u\|_{L^{\infty}(\overline{\Omega})}\geq\left(\frac{\sqrt{2n}}{diam\,\Omega}\right)^{\frac{2m}{p-1}}.
\end{equation*}
\end{thm}

It's well known that the super poly-harmonic properties of solutions are crucial in establishing Liouville type theorems and the representation formulae for higher order or fractional order PDEs (see e.g. \cite{CDQ,CF,CFL,CL2,WX}). In Section 2, we will first prove the super poly-harmonic properties of solutions for both $a<0$ and $a\geq0$ via a unified approach (see Theorem \ref{lemma0}). As a consequence, we can show the equivalence between the PDE \eqref{PDE} and the corresponding integral equation \eqref{IE}. Then, by applying the method of moving planes in integral forms and Pohozaev identity, we prove the Liouville theorem (Theorem \ref{Thm0}) for \eqref{PDE}. In Sections 3 and 4, we will prove a priori estimates and existence of positive solutions to non-critical higher order Lane-Emden equations in bounded domains $\Omega$, using the arguments from Chen, Dai and Qin \cite{CDQ} for critical order Hardy-H\'{e}non equations and results from Chen, Fang and Li \cite{CFL}. In Section 3, we will derive a priori estimates for any positive solutions to the higher order Naiver problem \eqref{tNavier} (Theorem \ref{Thm1}) by applying the method of moving planes in local way and Kelvin transforms. We will first establish the boundary layer estimates (Theorem \ref{Boundary}), in which the properties of the boundary $\partial\Omega$ play a crucial role. The global a priori estimates follows from the boundary layer estimates, blowing-up analysis and the Liouville theorem (Theorem \ref{Thm0}). Section 4 is devoted to the proof of Theorem \ref{Thm2}. The existence of positive solutions to the higher order Lane-Emden equations \eqref{Navier} with Navier boundary conditions will be established via the a priori estimates (Theorem \ref{Thm-CFL} and Theorem \ref{Thm1}) and the Leray-Schauder fixed point theorem (Theorem \ref{L-S}).

\section{Proof of Theorem \ref{Thm0}}

In this section, we will prove Theorem \ref{Thm0} by using contradiction arguments. Now suppose on the contrary that $u\geq0$ satisfies equation \eqref{PDE} but $u$ is not identically zero, then there exists some $\bar{x}\in\mathbb{R}^{n}$ such that $u(\bar{x})>0$.

In the following, we will use $C$ to denote a general positive constant that may depend on $n$, $m$, $a$, $p$ and $u$, and whose value may differ from line to line.

\subsection{Super poly-harmonic properties}
The super poly-harmonic properties of solutions are closely related to the representation formulae and Liouville type theorems (see \cite{CDQ,CF,CFL,CL2,WX} and the references therein). Therefore, in order to prove Theorem \ref{Thm0}, we need the following theorem about the super poly-harmonicity.

\begin{thm}\label{lemma0}(Super poly-harmonic properties). Assume $n\geq3$, $1\leq m<\frac{n}{2}$, $-\infty<a<2m$, $1<p<+\infty$ and $u$ is a nonnegative solution of \eqref{PDE}. If one of the following two assumptions
\begin{equation*}
  -\infty<a\leq2+2p \,\,\,\,\,\,\,\,\,\,\,\, \text{or} \,\,\,\,\,\,\,\,\,\,\,\, u(x)=o(|x|^{2}) \,\,\,\, \text{as} \,\, |x|\rightarrow+\infty
\end{equation*}
holds, then
\begin{equation*}
  (-\Delta)^{i}u(x)\geq0
\end{equation*}
for every $i=1,2,\cdots,m-1$ and all $x\in\mathbb{R}^{n}$.
\end{thm}
\begin{proof}
Let $u_{i}:=(- \Delta)^{i}u$. We want to show that $u_{i}\geq0$ for $i=1,2,\cdots, m-1$. Our proof will be divided into two steps.

\textbf{\emph{Step 1.}} We first show that
\begin{equation}\label{2-1}
u_{ m-1}=(-\Delta)^{ m- 1}u\geq0.
\end{equation}
If not, then there exists $0\neq x^{1}\in\mathbb{R}^n$, such that
\begin{equation}\label{2-2}
  u_{ m-1}(x^{1})<0.
\end{equation}

Now, let
\begin{equation}\label{2-3}
  \bar{f}(r)=\bar{f}\big(|x-x^1|\big):=\frac{1}{|\partial B_{r}(x^{1})|}\int_{\partial B_{r}(x^{1})}f(x)d\sigma
\end{equation}
be the spherical average of $f$ with respect to the center $x^1$. Then, by the well-known property $\overline{\Delta u}=\Delta\bar{u}$ and $-\infty<a<2m<n$, we have, for any $r\geq0$ and $r\neq|x^{1}|$,
\begin{equation}\label{2-4}
\left\{{\begin{array}{l} {-\Delta\overline{u_{ m-1}}(r)=\overline{\frac{u^{p}(x)}{|x|^{a}}}(r)}, \\  {} \\ {-\Delta\overline{u_{ m-2}}(r)=\overline{u_{ m-1}}(r)}, \\ \cdots\cdots \\ {-\Delta\overline u(r)=\overline{u_1}(r)}. \\ \end{array}}\right.
\end{equation}
From the first equation in \eqref{2-4}, by Jensen's inequality, we get, for any $r\geq0$ and $r\neq|x^{1}|$,
\begin{align}\label{2-5}
-\Delta\overline{u_{ m-1}}(r)&=\frac{1}{{| {\partial
B_{r}({x^{1}})}| }}\int_{\partial B_{r}(
{x^{1}})}\frac{{u^{p}(x)}}{|x|
^{a}}d\sigma\nonumber\\
& \geq({r+| {x^{1}}| })^{-a}\frac
{1}{{| {\partial B_{r}({x^{1}})}| }}
\int_{\partial B_{r}({x^{1}})}{u^{p}(  x)}d\sigma\\
&  \geq({r+| {x^{1}}| })^{-a}\left(
{\frac{1}{{| {\partial B_{r}({x^{1}})}| }
}\int_{\partial B_{r}({x^{1}})}{u(x)
}d\sigma}\right)^{p}\nonumber\\
& =(r+|x^{1}|)^{-a}\bar{u}^{p}(r)\geq0 \quad\quad \text{if} \,\,\, 0\leq a<2m,\nonumber
\end{align}
and
\begin{equation}\label{2-5'}
  -\Delta\overline{u_{ m-1}}(r)\geq\big|r-|x^{1}|\big|^{-a}\bar{u}^{p}(r)\geq0 \quad\quad \text{if} \,\, -\infty<a<0.
\end{equation}
From \eqref{2-5} and \eqref{2-5'}, one has
\begin{equation}\label{2-6}
  -\frac{1}{r^{n-1}}\Big(r^{n-1}\overline{u_{ m-1}}\,'(r)\Big)'\geq0.
\end{equation}
Since $-\infty<a<2m<n$, we can integrate both sides of \eqref{2-6} from $0$ to $r$ and derive
\begin{equation}\label{2-7}
\overline{u_{ m-1}}\,'(r)\leq0, \,\,\,\,\,\, \overline{u_{ m-1}}(r)\leq\overline{u_{ m-1}}(0)=u_{ m-1}(x^{1})=:-c_{0}<0
\end{equation}
for any $r\geq0$. From the second equation in \eqref{2-4}, we deduce that
\begin{equation}\label{2-8}
-\frac{1}{{r^{n-1}}}\Big({r^{n-1}\overline{u_{ m-2}}\,'}(r)\Big)'=\overline{u_{ m-1}}(r)\leq-c_{0}, \,\,\,\,\,\, \forall \,\, r\geq0,
\end{equation}
integrating from $0$ to $r$ yields
\begin{equation}\label{2-9}
\overline{u_{ m - 2}}\,'(r)\geq\frac{c_{0}}{n}r, \,\,\,\,\,\, \overline{u_{ m-2}}(r)\geq\overline{u_{ m-2}}(0)+\frac{c_{0}}{2n}r^{2}, \,\,\,\,\,\, \forall \,\, r\geq0.
\end{equation}
Hence, there exists $r_{1} > 0$ such that
\begin{equation}\label{2-10}
  \overline{u_{ m-2}}(r_{1})>0.
\end{equation}
Next, take a point $x^{2}$ with $|x^{2}-x^{1}|=r_{1}$ as the new center, and make average of $\bar{f}$ at the new center $x^{2}$, i.e.,
\begin{equation}\label{2-11}
\overline{\overline{f}}(r)=\overline{\overline{f}}\big(|x-x^{2}|\big):=\frac{1}{|\partial B_{r}(x^{2})|}\int_{\partial B_{r}(x^{2})}\bar f(x)d\sigma.
\end{equation}
One can easily verify that
\begin{equation}\label{2-12}
\overline{\overline{u_{ m-2}}}(0)=\overline{u _{ m - 2}}(x^{2})=:c_{1}>0.
\end{equation}
Then, from \eqref{2-5} and Jensen's inequality, we deduce that $(\overline{\overline{u}},\overline{\overline{u_{1}}},\cdots,\overline{\overline{u_{ m-1}}})$ satisfies
\begin{equation}\label{2-13}
\left\{{\begin{array}{l} {-\Delta\overline{\overline{u_{ m-1}}}(r)=\overline{\overline{\frac{u^{p}(x)}{|x|^{a}}}}(r)\geq0}, \\
{} \\
{-\Delta\overline{\overline{u_{ m-2}}}(r)=\overline{\overline{u_{ m-1}}}(r)}, \\ \cdots\cdots \\ {-\Delta\overline{\overline{u}}(r)=\overline{\overline{u _1}}(r)} \\ \end{array}}\right.
\end{equation}
for any $r\geq0$. Using the same method as obtaining the estimate \eqref{2-9}, we conclude that
\begin{equation}\label{2-14}
  \overline{\overline{u_{ m-2}}}(r)\geq\overline{\overline{u_{ m-2}}}(0)+\frac{{c_{0}}}{{2n}}r^{2}, \,\,\,\,\,\, \forall \,\, r\geq0.
\end{equation}
Thus we infer from \eqref{2-7}, \eqref{2-12}, \eqref{2-13} and \eqref{2-14} that
\begin{equation}\label{2-15}
\overline{\overline{u_{ m-1}}}(r)\leq\overline{\overline{u_{ m-1}}}(0)<0, \,\,\,\,\,\,\,\,\, \overline{\overline{u_{ m-2}}}(r)\geq\overline{\overline{u_{ m-2}}}(0)>0, \,\,\,\,\,\, \forall \,\, r\geq0.
\end{equation}
From the third equation in \eqref{2-13} and integrating, we infer that
\begin{equation}\label{2-16}
  \overline{\overline{u_{ m-3}}}\,'(r)\leq-\frac{c_{1}}{n}r \,\,\,\,\,\, \text{and} \,\,\,\,\,\, \overline{\overline{u_{ m-3}}}(r)\leq\overline{\overline{u_{ m-3}}}(0)-\frac{{c_{1}}}{{2n}}r^{2}, \,\,\,\,\,\, \forall \,\, r\geq0.
\end{equation}
Hence, there exists $r_{2}>0$ such that
\begin{equation}\label{2-17}
  \overline{\overline{u_{ m-3}}}(r_{2})<0.
\end{equation}
Next, we take a point $x^{3}$ with $|x^{3}-x^{2}|=r_{2}$ as the new center and make average of $\bar{\bar{f}}$ at the new center $x^{3}$, i.e.,
\begin{equation}\label{2-18}
\overline{\overline{\overline{f}}}(r)=\overline{\overline{\overline{f}}}\big(|x-x^{3}|\big):=\frac{1}{|\partial B_{r}(x^{3})|}\int_{\partial B_{r}(x^{3})}\overline{\overline{f}}(x)d\sigma.
\end{equation}
It follows that
\begin{equation}\label{2-19}
\overline{\overline{\overline{u_{ m-3}}}}(0)=\overline{\overline{u _{ m-3}}}(x^{3})=:-c_{2}<0.
\end{equation}
One can easily verify that $\overline{\overline{\overline{u}}}$ and $\overline{\overline{\overline{u_{i}}}}$ ($i=1,\cdots, m-1$) satisfy entirely similar equations as $(\overline{\overline{u}},\overline{\overline{u_{1}}},\cdots,\overline{\overline{u_{ m-1}}})$ (see \eqref{2-13}). Using the same method as deriving \eqref{2-15}, we arrive at
\begin{equation}\label{2-20}
\overline{\overline{\overline{u_{ m-1}}}}(r)\leq\overline{\overline{\overline{u_{ m-1}}}}(0)<0, \,\,\,\,\,\,\, \overline{\overline{\overline{u_{ m-2}}}}(r)\geq\overline{\overline{\overline{u_{ m-2}}}}(0)>0, \,\,\,\,\,\,\,
\overline{\overline{\overline{u_{ m-3}}}}(r)\leq\overline{\overline{\overline{u_{ m-3}}}}(0)<0
\end{equation}
for any $r\geq0$. Continuing this way, after $ m$ steps of re-centers (denotes the centers by $x^{1},x^{2},\cdots,x^{ m}$, the $ m$ times averages of $f$ by $\widetilde{f}$ and the resulting functions coming from taking $ m$ times averages by $\widetilde{u}$ and $\widetilde{u_{i}}$ for $i=1,2,\cdots, m-1$), we finally obtain that
\begin{equation}\label{2-21}
-\Delta\widetilde{u_{ m-1}}(r)\geq\widetilde{\frac{u^{p}(x)}{|x|^{a}}}(r)\geq0,
\end{equation}
and for every $i=1,\cdots, m-1$,
\begin{equation}\label{2-22}
(-1)^{i}\widetilde{u_{ m-i}}(r)\geq(-1)^{i}\widetilde{u_{ m-i}}(0)>0, \,\,\, \,\,\, (-1)^{ m}\widetilde{u}(r)\geq(-1)^{ m}\widetilde{u}(0)>0, \,\,\,\,\,\, \forall \,\, r\geq0.
\end{equation}
Moreover, in the above process, we may choose $|x^{ m}|$ sufficiently large, such that
\begin{equation}\label{2-100}
  |x^{ m}-x^{ m-1}|\geq|x^{ m-1}-x^{ m-2}|+\cdots+|x^{2}-x^{1}|+|x^{1}|+2.
\end{equation}

Now, if $ m$ is odd, estimate \eqref{2-22} implies immediately that
\begin{equation}\label{2-23}
  \widetilde{u}(r)\leq\widetilde{u}(0)<0,
\end{equation}
which contradicts the fact that $u\geq0$. Therefore, we only need to deal with the cases that $ m$ is an even integer hereafter.

Since $ m$ is even, we have $\widetilde{u}(r)\geq\widetilde{u}(0)>0$ for any $r\geq0$, furthermore, one can actually observe from the above ``re-centers and iteration" process that
\begin{equation}\label{2-101}
  \widetilde{u}(0)\geq\frac{c}{2n}|x^{ m}-x^{ m-1}|^{2}
\end{equation}
for some constant $c>0$. Thus we may choose $|x^{ m}|$ larger, such that both \eqref{2-100} and the following
\begin{equation}\label{2-102}
  \widetilde{u}(0)\geq(2p)^{\frac{2mp}{(p-1)^{2}}}\left(1+\frac{2n}{p}\right)^{\frac{2m}{p-1}}
\end{equation}
hold.

For arbitrary $\lambda>0$, define the re-scaling of $u$ by
\begin{equation}\label{2-24}
  u_{\lambda}(x):=\lambda^{\frac{2m-a}{p-1}}u(\lambda x).
\end{equation}
Then one can easily verify that equation \eqref{PDE} is invariant under this re-scaling. After $ m$ steps of re-centers for $u_{\lambda}$, we denote the centers for $u_{\lambda}$ by $x_{\lambda}^{1},x_{\lambda}^{2},\cdots,x_{\lambda}^{ m}$ and the resulting function coming from taking $ m$ times averages by $\widetilde{u_{\lambda}}$ and $\widetilde{u_{\lambda,i}}$ for $i=1,2,\cdots, m-1$. Then \eqref{2-21} and \eqref{2-22} still hold for $(\widetilde{u_{\lambda}},\widetilde{u_{\lambda,1}},\cdots,\widetilde{u_{\lambda, m-1}})$ and $x_{\lambda}^{k}=\frac{1}{\lambda}x_{k}$ for $k=1,\cdots, m$, thus one has the following estimate
\begin{equation}\label{2-25}
  |x_{\lambda}^{ m}-x_{\lambda}^{ m-1}|+\cdots+|x_{\lambda}^{2}-x_{\lambda}^{1}|+|x_{\lambda}^{1}|\leq
|x^{ m}-x^{ m-1}|+\cdots+|x^{2}-x^{1}|+|x^{1}|=:M
\end{equation}
holds uniformly for every $\lambda\geq1$.

Since we have \eqref{2-22} and $m$ is even, it follows that
\begin{equation}\label{2-26}
  \widetilde{u}(r)\geq\widetilde{u}(0)\geq(2p)^{\frac{2mp}{(p-1)^{2}}}\left(1+\frac{2n}{p}\right)^{\frac{2m}{p-1}}>0, \,\,\,\,\,\, \forall \,\, r\geq0,
\end{equation}
and hence
\begin{equation}\label{2-28}
\widetilde{u_{\lambda}}(r)\geq\widetilde{u_{\lambda}}(0)=\lambda^{\frac{2m-a}{p-1}}\widetilde{u}(0)
\geq\lambda^{\frac{2m-a}{p-1}}(2p)^{\frac{2mp}{(p-1)^{2}}}\left(1+\frac{2n}{p}\right)^{\frac{2m}{p-1}}>0, \,\,\,\,\,\,\,\,\, \forall \,\, r\geq0.
\end{equation}
For $0\leq a<2m$, by the estimate \eqref{2-28}, we may assume that, we already have
\begin{equation}\label{2-28'}
\widetilde{u}(0)\geq(1+M)^{\frac{a}{p-1}}(2p)^{\frac{2mp}{(p-1)^{2}}}\left(1+\frac{2n}{p}\right)^{\frac{2m}{p-1}},
\end{equation}
or else we may replace $u$ by $u_{\lambda}$ with $\lambda=(1+M)^{\frac{a}{2m-a}}$ (still denoted by $u$).

For any $0\leq r\leq1$, we have
\begin{equation}\label{2-27}
  \widetilde{u}(r)\geq\widetilde{u}(0)\geq l_{0}\,r^{\alpha_{0}},
\end{equation}
where
\begin{equation}\label{2-29}
  l_{0}:=\widetilde{u}(0)\geq\max\left\{(1+M)^{\frac{a}{p-1}},1\right\}(2p)^{\frac{2mp}{(p-1)^{2}}}\alpha_{0}^{\frac{2m}{p-1}}, \,\quad\, \alpha_{0}:=\max\Big\{1,\frac{2n}{p}\Big\}\geq1.
\end{equation}
As a consequence, we infer from \eqref{2-21}, \eqref{2-100}, \eqref{2-25} and \eqref{2-27} that, for any $0\leq r\leq1$,
\begin{align}\label{2-30}
-\Delta\widetilde{u_{ m-1}}(r)&\geq\Big(r+|x^{ m}-x^{ m-1}|+\cdots+|x^{2}-x^{1}|+|x^{1}|\Big)^{-a}\widetilde{u}^{p}(r) \nonumber \\
&\geq\big(1+M\big)^{-a}\,l_{0}^{p}\,r^{\alpha_{0}p}\\
&\geq C_{0}\,l_{0}^{p}\,r^{\alpha_{0}p} \,\,\,\,\,\,\,\,\,\,\,\,\quad \text{if} \,\, 0\leq a<2m,  \nonumber
\end{align}
and
\begin{align}\label{2-30'}
-\Delta\widetilde{u_{ m-1}}(r)&\geq\Big(|x^{ m}-x^{ m-1}|-|x^{ m-1}-x^{ m-2}|-\cdots-|x^{2}-x^{1}|-|x^{1}|-r\Big)^{-a}
\widetilde{u}^{p}(r) \nonumber \\
&\geq l_{0}^{p}\,r^{\alpha_{0}p}\\
&\geq C_{0}\,l_{0}^{p}\,r^{\alpha_{0}p}, \,\,\,\,\,\,\,\,\,\,\,\,\quad \text{if} \,\, -\infty<a<0,  \nonumber
\end{align}
where
\begin{equation}\label{2-99}
  C_{0}:=\min\left\{(1+M)^{-a},1\right\}\in(0,1].
\end{equation}
Integrating both sides of \eqref{2-30} and \eqref{2-30'} from $0$ to $r$ twice and taking into account of \eqref{2-22} yield
\begin{equation}\label{2-31}
  \widetilde{u_{ m-1}}(r)<-\frac{C_{0}l_{0}^{p}}{(\alpha_{0}p+n)(\alpha_{0}p+2)}r^{\alpha_{0}p+2}, \,\,\,\,\,\, \forall \,\, 0\leq r\leq1.
\end{equation}
This implies
\begin{equation}\label{2-32}
  -\frac{1}{r^{n-1}}\left(r^{n-1}\widetilde{u_{ m-2}}\,'(r)\right)'<-\frac{C_{0}l_{0}^{p}}{(\alpha_{0}p+n)(\alpha_{0}p+2)}r^{\alpha_{0}p+2},
\end{equation}
and consequently,
\begin{equation}\label{2-33}
  \widetilde{u_{ m-2}}(r)>\frac{C_{0}l_{0}^{p}}{(\alpha_{0}p+n)(\alpha_{0}p+2)(\alpha_{0}p+n+2)(\alpha_{0}p+4)}r^{\alpha_{0}p+4}, \,\,\,\,\,\, \forall \,\, 0\leq r\leq1.
\end{equation}
Continuing this way, since $m$ is an even integer, by iteration, we can finally arrive at
\begin{equation}\label{2-34}
  \widetilde{u}(r)>\frac{C_{0}l_{0}^{p}}{(\alpha_{0}p+2n)^{2m}}r^{\alpha_{0}p+2m}, \,\,\,\,\,\, \forall \,\, 0\leq r\leq1.
\end{equation}
Now, define
\begin{equation}\label{2-35}
  \alpha_{k+1}:=2\alpha_{k}p\geq\alpha_{k}p+2n \,\,\,\,\,\, \text{and} \,\,\,\,\,\, l_{k+1}:=\frac{C_{0}l_{k}^{p}}{(2\alpha_{k}p)^{2m}}
\end{equation}
for $k=0,1,\cdots$. Then \eqref{2-34} implies
\begin{equation}\label{2-36}
\widetilde{u}(r)>\frac{C_{0}l_{0}^{p}}{(2\alpha_{0}p)^{2m}}r^{2\alpha_{0}p}=l_{1}r^{\alpha_{1}}, \,\,\,\,\, \forall \,\, r\in[0,1].
\end{equation}
Suppose we have $\widetilde{u}(r)\geq l_{k}r^{\alpha_{k}}$, then go through the entire process as above, we can derive $\widetilde{u}(r)\geq l_{k+1}r^{\alpha_{k+1}}$ for any $0\leq r\leq1$. Therefore, one can prove by induction that
\begin{equation}\label{2-37}
  \widetilde{u}(r)\geq l_{k}r^{\alpha_{k}}, \,\,\,\,\,\, \forall \,\, r\in[0,1], \,\,\,\,\,\, \forall \,\, k\in\mathbb{N}.
\end{equation}
Through direct calculations, we have
\begin{eqnarray}\label{2-38}
  l_{k}&=&\frac{C_{0}^{\frac{p^{k}-1}{p-1}}l_{0}^{p^{k}}}{(2p)^{2m(k+(k-1)p+(k-2)p^{2}+\cdots+p^{k-1})}\alpha_{0}^{\frac{2m(p^{k}-1)}{p-1}}} \\
 \nonumber &=& \frac{C_{0}^{\frac{p^{k}-1}{p-1}}l_{0}^{p^{k}}(2p)^{\frac{2mk}{p-1}}}{(2p)^{\frac{2m(p^{k+1}-p)}{(p-1)^{2}}}\alpha_{0}^{\frac{2m(p^{k}-1)}{p-1}}}
\geq (2p)^{\frac{2mk}{p-1}}\left(\frac{C_{0}^{\frac{1}{p-1}}l_{0}}{(2p)^{\frac{2mp}{(p-1)^{2}}}\alpha_{0}^{\frac{2m}{p-1}}}\right)^{p^{k}}
\end{eqnarray}
for $k=0,1,2,\cdots$. From \eqref{2-29}, \eqref{2-99}, \eqref{2-37} and \eqref{2-38}, we deduce that
\begin{equation}\label{2-40}
  \widetilde{u}(1)\geq(2p)^{\frac{2mk}{p-1}}\rightarrow+\infty, \,\,\,\,\,\, \text{as} \,\, k\rightarrow\infty.
\end{equation}
This is absurd. Therefore, \eqref{2-1} must hold, that is, $u_{ m-1}=(-\Delta)^{ m- 1}u\geq0$.

\textbf{\emph{Step 2.}} Next, we will show that all the other $u_{i}$ ($i=1,\cdots, m-2$) must be nonnegative, that is,
\begin{equation}\label{2-41}
u_{ m-i}(x)\geq0, \,\,\,\,\,\,\,\,\,\,\, \forall \,\, i=2,3,\cdots, m-1, \,\,\,\,\,\, \forall \,\, x\in\mathbb{R}^{n}.
\end{equation}
Suppose on the contrary that, there exists some $2\leq i\leq m-1$ and some $x^{0}\in\mathbb{R}^{n}$ such that
\begin{equation}\label{2-42}
  u_{ m-1}(x)\geq0, \,\,\,\,\, u_{ m-2}(x)\geq0, \,\,\,\, \cdots, \,\,\,\, u_{ m-i+1}(x)\geq0, \,\,\,\,\,\, \forall \,\, x\in\mathbb{R}^{n},
\end{equation}
\begin{equation}\label{2-43}
  u_{ m-i}(x^{0})<0.
\end{equation}
Then, repeating the similar ``re-centers and iteration" arguments as in Step 1, after $ m-i+1$ steps of re-centers (denotes the centers by $\bar{x}^{1},\bar{x}^{2},\cdots,\bar{x}^{ m-i+1}$), the signs of the resulting functions $\widetilde{u_{ m-j}}$ ($j=i,\cdots, m-1$) and $\widetilde{u}$ satisfy
\begin{equation}\label{2-44}
  (-1)^{j-i+1}\widetilde{u_{ m-j}}(r)\geq(-1)^{j-i+1}\widetilde{u_{ m-j}}(0)>0, \,\,\,\,\,\,
(-1)^{ m-i+1}\widetilde{u}(r)\geq(-1)^{ m-i+1}\widetilde{u}(0)>0
\end{equation}
for any $r\geq0$. Since $u\geq0$, it follows immediately from \eqref{2-44} that $ m-i+1$ is even and
\begin{equation}\label{2-45}
  \widetilde{u}(r)\geq\widetilde{u}(0)>0, \,\,\,\,\,\, \forall \,\, r\geq0.
\end{equation}
Furthermore, since $ m-i$ is odd, we infer from \eqref{2-44} that
\begin{equation}\label{2-48}
  -\Delta\widetilde{u}(r)=\widetilde{u_{1}}(r)\leq\widetilde{u_{1}}(0)=:-\widetilde{c}<0, \,\,\,\,\,\, \forall \,\, r\geq0,
\end{equation}
and hence, by integrating, one has
\begin{equation}\label{2-49}
  \widetilde{u}(r)\geq\widetilde{u}(0)+\frac{\widetilde{c}}{2n}r^{2}>\frac{\widetilde{c}}{2n}r^{2}, \,\,\,\,\,\, \forall \,\, r\geq0.
\end{equation}
Therefore, if we assume that $u(x)=o(|x|^{2})$ as $|x|\rightarrow+\infty$, we will get a contradiction from \eqref{2-49}.

Or, if we assume that $-\infty<a\leq2+2p$, combining \eqref{2-49} with the estimate \eqref{2-21}, we get that, for $r\geq r_{0}$ sufficiently large,
\begin{eqnarray}\label{2-46}
  -\Delta\widetilde{u_{ m-1}}(r)&\geq&\left(r+|\bar{x}^{ m-i+1}-\bar{x}^{ m-i}|+\cdots+|\bar{x}^{2}-\bar{x}^{1}|+|\bar{x}^{1}|
\right)^{-a}\widetilde{u}^{p}(r) \\
  \nonumber &\geq&\left(\frac{\widetilde{c}}{4n}\right)^{p}r^{2p-a} \quad\quad\quad \text{if} \,\, 0\leq a\leq2+2p,
\end{eqnarray}
and
\begin{eqnarray}\label{2-46'}
  -\Delta\widetilde{u_{ m-1}}(r)&\geq&\left(r-|\bar{x}^{ m-i+1}-\bar{x}^{ m-i}|-\cdots-|\bar{x}^{2}-\bar{x}^{1}|-|\bar{x}^{1}|
\right)^{-a}\widetilde{u}^{p}(r) \\
  \nonumber &\geq&\left(\frac{\widetilde{c}}{4n}\right)^{p}r^{2p-a} \quad\quad\quad \text{if} \,\, -\infty<a<0.
\end{eqnarray}
Now, by a direct integration on \eqref{2-46} and \eqref{2-46'}, we get, if $-\infty<a<2+2p$, then
\begin{equation}\label{2-47}
\widetilde{u_{ m-1}}(r)\leq\widetilde{u_{ m-1}}(r_{0})-\left(\frac{\widetilde{c}}{4n}\right)^{p}
\frac{r^{2+2p-a}-r_{0}^{2+2p-a}}{(n+2p-a)(2+2p-a)}\rightarrow-\infty, \,\,\,\,\,\, \text{as} \,\,\, r\rightarrow\infty;
\end{equation}
if $a=2+2p$, then
\begin{equation}\label{2-47'}
\widetilde{u_{ m-1}}(r)\leq\widetilde{u_{ m-1}}(r_{0})-\left(\frac{\widetilde{c}}{4n}\right)^{p}\frac{\ln r-\ln r_{0}}{n-2}\rightarrow-\infty, \,\,\,\,\,\, \text{as} \,\,\, r\rightarrow\infty.
\end{equation}
This contradicts $u_{ m-1}\geq0$ and thus \eqref{2-41} must hold. This concludes the proof of Theorem \ref{lemma0}.
\end{proof}

\begin{rem}\label{remark4}
For $0<a<2m$, if we consider the nonnegative solutions $u\in C^{2m}(\mathbb{R}^{n}\setminus\{0\})\cap C(\mathbb{R}^{n})$, then it is clear from our proof of Theorem \ref{lemma0} that we can show super poly-harmonic properties except the origin $0\in\mathbb{R}^{n}$, that is, $(-\Delta)^{i}u\geq0$ in $\mathbb{R}^{n}\setminus\{0\}$ for $i=1,\cdots,m-1$.
\end{rem}

\subsection{Equivalance between PDE and IE}
By applying Theorem \ref{lemma0} for $a\geq0$, we can deduce from $-\Delta u\geq0$, $u\geq0$, $u(\bar{x})>0$ and maximum principle that
\begin{equation}\label{2-50}
  u(x)>0, \,\,\,\,\,\,\, \forall \,\, x\in\mathbb{R}^{n}.
\end{equation}
Then, by maximum principle, Lemma 2.1 from Chen and Lin \cite{CLin} and induction, we can also infer further from $(-\Delta)^{i} u\geq0$ ($i=1,\cdots, m-1$), $u>0$ and equation \eqref{PDE} that
\begin{equation}\label{2-51}
  (-\Delta)^{i}u(x)>0, \,\,\,\,\,\,\,\, \forall \,\, i=1,\cdots, m-1, \,\,\,\, \forall \,\, x\in\mathbb{R}^{n}.
\end{equation}

Next, we will show that the positive solution $u$ to \eqref{PDE} also satisfies the following integral equation
\begin{equation}\label{IE}
  u(x)=\int_{\mathbb{R}^{n}}\frac{C}{|x-\xi|^{n-2m}}\cdot\frac{u^{p}(\xi)}{|\xi|^{a}}d\xi.
\end{equation}

Indeed, we have the following theorem on the equivalence between PDE \eqref{PDE} and IE \eqref{IE}.
\begin{thm}\label{equivalence}
Assume $n\geq3$, $1\leq m<\frac{n}{2}$, $0\leq a<2m$ and $1<p<\infty$. Suppose $u$ is nonnegative classical solution to \eqref{PDE}, then it also solves the integral equation \eqref{IE}, and vice versa.
\end{thm}
\begin{proof}
Let ${\delta}(x-{\xi})$ be the Dirac Delta function and $\phi_{r}(x-{\xi})$ be the solution of the following equation
\begin{equation}\begin{cases}\label{2e1}
(-\Delta)^{m}\phi_{r}(x-{\xi})={\delta}(x-{\xi}),\ \ \ \ {\xi}\in B_{r}(x),\\
\phi_{r}(x-{\xi})=(-\Delta)\phi_{r}(x-{\xi})=...=(-\Delta)^{m-1}\phi_{r}(x-{\xi}), \,\,\,\,\, {\xi}\in {\partial}B_{r}(x).
\end{cases}\end{equation}
One can easily verify that, $(-\Delta)^{i}\phi_{r}(x-{\xi})$ must take the following form
\begin{equation}\label{2e2}
  (-\Delta)^{i}\phi_{r}(x-{\xi})=\frac{c_{i}}{|x-{\xi}|^{n+2i-2m}}+\sum_{k=1}^{m-i}c_{i,k}\frac{|x-{\xi}|^{2m-2i-2k}}{{r}^{n-2k}}
\end{equation}
for $i=0,1,\cdots,m-1$, where the coefficients satisfy $c_{i}+\sum_{k=1}^{m-i}c_{i,k}=0$ ($i=0,1,\cdots,m-1$). In particular, when $i=m-1$, by \eqref{2e2}, we have
\begin{equation}\label{2e3}
  (-\Delta)^{m-1}\phi_{r}(x-{\xi})=\frac{C_{m-1}}{|x-{\xi}|^{n-2}}-\frac{C_{m-1}}{{r}^{n-2}}, \,\,\,\,\,\,\, {\xi}\in \overline{B_{r}(x)},
\end{equation}
and hence
\begin{equation}\label{2e4}
\frac{\partial\left[(-\Delta)^{m-1}\phi_{r}(x-{\xi})\right]}{\partial v_{\xi}}\leqslant 0, \,\,\,\,\,\,\,\, {\xi}\in {\partial}B_{r}(x),
\end{equation}
where $v_{\xi}$ denotes the unit outer normal vector at $\xi\in\partial B_{r}(x)$. Next we define function $f(\xi)$ by
\begin{equation}\label{2e5}
(-\Delta)^{m-1}\phi_{r}(x-{\xi})=\frac{C_{m-1}}{|x-{\xi}|^{n-2}}-\frac{C_{m-1}}{{r}^{n-2}}=:f({\xi})\geqslant 0.
\end{equation}
It is obvious that $f\in {L}^{1}(B_{r}(x))$, thus $(-\Delta)^{m-2}\phi_{r}(x-{\xi})$ is super-harmonic in the sense of distribution in ${B_{r}(x)}$, and hence we derive
\begin{equation}\label{2e6}
\inf_{\xi\in B_{r}(x)}(-\Delta)^{m-2}\phi_{r}(x-{\xi})\geq\inf_{\xi\in\partial B_{r}(x)}(-\Delta)^{m-2}\phi_{r}(x-{\xi})=0
\end{equation}
and
\begin{equation}\label{2e7}
\frac{\partial{(-\Delta)^{m-2}\phi_{r}(x-{\xi})} }{\partial {v}_{\xi}}\leqslant 0, \,\,\,\,\,\,\,\, {\xi}\in {\partial}B_{r}(x).
\end{equation}
Continuing this way, we conclude that, for $i=0,1,2...m-1$,
\begin{equation}\label{2e8}
\inf_{\xi\in B_{r}(x)}(-\Delta)^{i}\phi_{r}(x-{\xi})\geq\inf_{\xi\in\partial B_{r}(x)}(-\Delta)^{i}\phi_{r}(x-{\xi})=0
\end{equation}
and
\begin{equation}\label{2e9}
\frac{\partial\left[(-\Delta)^{i}\phi_{r}(x-{\xi})\right]}{\partial {v}_{\xi}}\leqslant 0, \,\,\,\,\,\,\,\, {\xi}\in {\partial}B_{r}(x).
\end{equation}

From \eqref{2e3}, we can get $(-\Delta)^{m-1}\phi_{r}(x-{\xi})$ monotone increases about $r$ and tends to ${\frac{C_{m-1}}{|x-{\xi}|^{n-2}}}$ as $r\rightarrow+\infty$. As a consequence, we arrive at, for any ${r}_{2}>{r}_{1}>0$,
\begin{equation}\label{2e10}
(-\Delta)[(-\Delta)^{m-2}\phi_{{r}_{2}}(x-{\xi})-(-\Delta)^{m-2}\phi_{{r}_{1}}(x-{\xi})]\geq0,\, \, \, \, {\xi}\in B_{r_{1}}(x),
\end{equation}
and
\begin{equation}\label{2e11}
0=(-\Delta)^{m-2}\phi_{{r}_{1}}(x-{\xi})\leq(-\Delta)^{m-2}\phi_{{r}_{2}}(x-{\xi}),\, \, \,\, {\xi}\in {\partial}B_{r_{1}}(x).
\end{equation}
By maximum principle, we deduce that
\begin{equation}\label{2e12}
(-\Delta)^{m-2}\phi_{{r}_{2}}(x-{\xi})\geq(-\Delta)^{m-2}\phi_{{r}_{1}}(x-{\xi}), \,\,\,\,\,\, \forall \,\, \xi\in\mathbb{R}^{n}.
\end{equation}
So $(-\Delta)^{m-2}\phi_{r}(x-{\xi})$ also monotone increases about $r$ and tends to ${\frac{C_{m-2}}{|x-{\xi}|^{n-4}}}$ as $r\rightarrow+\infty$. Continuing this way, we can derive
\begin{equation}\label{2e13}
  (-\Delta)^{i}\phi_{r}(x-{\xi})\,\,\big\uparrow\,\,\frac{C_{i}}{|x-{\xi}|^{n-2m+2i}}, \,\,\,\,\,\,\, \text{as} \,\, r\rightarrow+\infty.
\end{equation}

By Lemma 1 in \cite{LiC} and equation \eqref{PDE}, we have $(-\Delta)^{m-1}u$ solves the following equation
\begin{equation}\label{bocher}
  (-\Delta)^{m}u=\frac{u^{p}(x)}{|x|^{a}}+m\delta(0) \quad\quad \text{in} \,\,\, B_{\rho}(0)
\end{equation}
in the sense of distributions for arbitrary $\rho>0$, where $m\geq0$ and $\delta(0)$ is the Delta distribution concentrated at the origin. Since $u\in C(\mathbb{R}^{n})$, it follows that $m=0$. Therefore, multiplying both sides of \eqref{bocher} by $\phi_{r}(x-{\xi})$ and integrating by parts on $B_{r}(x)$, by Theorem \ref{lemma0} and \eqref{2e9}, one has
\begin{eqnarray}\label{2e14}
&&\int_{B_{r}(x)}{\phi_{r}(x-{\xi})}\frac{u^{p}(\xi)}{|\xi|^{a}}d{\xi} \\
\nonumber &=& u(x)+\sum_{i=0}^{m-1}\int_{{\partial}B_{r}(x)}(-\Delta)^{i}u(\xi)\cdot\frac{\partial\left[(-\Delta)^{m-i-1}\phi_{r}(x-{\xi})\right]}{\partial {v}_{\xi}}d{\sigma}
\\
\nonumber &{\leqslant}& u(x)
\end{eqnarray}
for any $x\neq0$. At the same time, multiplying $(-\Delta)^{i}u$ by $(-\Delta)^{m-i}\phi_{r}(x-{\xi})$ ($i=1,\cdots,m-1$) and integrating by parts on $B_{r}(x)$, by Theorem \ref{lemma0} and \eqref{2e9}, one also has
\begin{eqnarray}\label{2e15}
&& \int_{B_{r}(x)}{(-\Delta)^{m-i}\phi_{r}(x-{\xi})}\cdot(-\Delta)^{i}u({\xi})d{\xi} \\
\nonumber &=& u(x)+\sum_{j=0}^{i-1}\int_{{\partial}B_{r}(x)}(-\Delta)^{j}u(\xi)\cdot\frac{\partial\left[(-\Delta)^{m-j-1}\phi_{r}(x-{\xi})\right]}{\partial {v}_{\xi}}d{\sigma}
\\
\nonumber &{\leqslant}& u(x).
\end{eqnarray}
Thus, by letting $r\rightarrow+\infty$ in \eqref{2e14}, \eqref{2e15} and using Levi's monotone convergence theorem, we obtain
\begin{equation}\label{2e16}
\int_{\mathbb{R}^{n}}\frac{1}{|x-\xi|^{n-2m}}\cdot\frac{u^{p}(\xi)}{|\xi|^{a}}d{\xi}<\infty
\end{equation}
and
\begin{equation}\label{2e17}
\int_{\mathbb{R}^{n}}\frac{(-\Delta)^{i}u(\xi)}{|x-{\xi}|^{n-2i}}<\infty
\end{equation}
for $i=1,\cdots,m-1$. Therefore, there exists a sequence $\{{r}_{k}\}$ such that, as ${r}_{k}\rightarrow\infty$,
\begin{equation}\label{2e18}
\frac{1}{{{r}_{k}}^{n-2m-1}}\int_{{\partial}B_{{r}_{k}}(x)}\frac{u^{p}(\xi)}{|\xi|^{a}}d{\sigma}\rightarrow0,
\end{equation}
and
\begin{equation}\label{2e19}
\frac{1}{{{r}_{k}}^{n-2i-1}}\int_{{\partial}B_{{r}_{k}}(x)}(-\Delta)^{i}{u}(\xi)d{\sigma}\rightarrow0\,\,\,\,\,\,\,\, \text{for} \,\, i=1,2,\cdots,m-1.
\end{equation}
From \eqref{2e18}, it follows that, as $r_{k}\rightarrow+\infty$,
\begin{equation}\label{2e20}
\frac{1}{{{r}_{k}}^{n-2m-1+a}}\int_{{\partial}B_{{r}_{k}}(x)}{{u}^{p}(\xi)}d{\sigma}=\frac{1}{{{r}_{k}}^{n-1-(2m-a)}}\int_{{\partial}B_{{r}_{k}}(x)}{{u}^{p}(\xi)}d{\sigma}\rightarrow0
\end{equation}
Then, by Jensen's inequality, we have
\begin{equation}\label{2e21}
\left({\frac{1}{{{r}_{k}}^{n-1-(2m-a)}}\int_{{\partial}B_{{r}_{k}}(x)}{{u}^{p}(\xi)}}d{\sigma}\right)^{\frac{1}{p}}\frac{1}{{{r}_{k}}^{\frac{2m-a}{p}}}
\geq\frac{1}{{{r}_{k}}^{n-1}}\int_{{\partial}B_{{r}_{k}}(x)}{{u}(\xi)}d{\sigma},
\end{equation}
and hence
\begin{equation}\label{2e22}
\frac{1}{{{r}_{k}}^{n-1}}\int_{{\partial}B_{{r}_{k}}(x)}{{u}(\xi)}d{\sigma}\rightarrow0.
\end{equation}
Combining this with \eqref{2e2} and \eqref{2e19} implies
\begin{equation}\label{2e23}
\sum_{i=0}^{m-1}\int_{{\partial}B_{{r}_{k}}(x)}(-\Delta)^{i}u(\xi)\cdot\frac{\partial\left[(-\Delta)^{m-i-1}\phi_{r_{k}}(x-{\xi})\right]}{\partial {v}_{\xi}}d{\sigma}\rightarrow0,
\end{equation}
inserting \eqref{2e23} into \eqref{2e14} and letting $r_{k}\rightarrow+\infty$, we derive immediately
\begin{equation}\label{2e24}
  u(x)=\int_{\mathbb{R}^{n}}\frac{C}{|x-{\xi}|^{n-2m}}\cdot\frac{u^{p}(\xi)}{|\xi|^{a}}d\xi,
\end{equation}
that is, $u$ satisfies the integral equation \eqref{IE}.

Conversely, assume that $u$ is a nonnegative classical solution of integral equation \eqref{IE}, then
\begin{eqnarray}\label{2e25}
(-\Delta)^{m}u(x)
\nonumber &=& \int_{\mathbb{R}^{n}}{\left[(-\Delta)^{m}\left(\frac{C}{|x-\xi|^{n-2m}}\right)\right]}\frac{u^{p}(\xi)}{|\xi|^{a}}d{\xi}
\\
\nonumber &=& \int_{\mathbb{R}^{n}}\delta(x-\xi)\frac{u^{p}(\xi)}{|\xi|^{a}}d{\xi}=\frac{u^{p}(x)}{|x|^{a}},
\end{eqnarray}
that is, $u$ also solves the PDE \eqref{PDE}. This completes the proof of equivalence between PDE \eqref{PDE} and IE \eqref{IE}.
\end{proof}

For $2\leq a<2m$ and $1<p<\infty$, one can easily observe that the regularity at $0$ of $u$ indicated by the integral equation \eqref{IE} contradicts with $u\in C^{2m-2}(\mathbb{R}^{n})$, thus we must have $u\equiv0$ in $\mathbb{R}^{n}$.

In the following, we will also obtain a contradiction for $1<p<\frac{n+2m-2a}{n-2m}$ and $0\leq a<2m$ by applying the method of moving planes and Pohozaev identity to the equivalent integral equation \eqref{IE} (see subsection 2.3 and 2.4). The proof still works for $u\in C^{2m}(\mathbb{R}^{n}\setminus\{0\})\cap C(\mathbb{R}^{n})$.

\subsection{Radial symmetry of positive solution}
From Theorem \ref{equivalence}, we know that the positive classical solution $u$ to PDE \eqref{PDE} is also a positive solution to the equivalent integral equation \eqref{IE}.

If $u$ is a nonnegative solution to IE \eqref{IE}, we must have either $u\equiv0$ or $u>0$ in $\mathbb{R}^{n}$. The next Theorem says that all the locally integrable positive solutions to IE \eqref{IE} must be radially symmetric and monotone decreasing about the origin.
\begin{thm}\label{symmetry}
Assume $n\geq3$, $1\leq m<\frac{n}{2}$, $0\leq a<2m$ and $1<p<\frac{n+2m-a}{n-2m}$. Suppose $u$ is a positive solution to IE \eqref{IE} satisfying $\frac{u^{p-1}}{|x|^{a}}\in L^{\frac{n}{2m}}_{loc}(\mathbb{R}^{n})$, then $u$ is radially symmetric and monotone decreasing about the origin.
\end{thm}
\begin{proof}
We define the Kelvin transform of $u$ by
\begin{equation}\label{s1}
  \bar{u}(x)=\frac{1}{|x|^{n-2m}}u\left(\frac{x}{|x|^{2}}\right), \,\,\,\,\,\,\,\, x\neq0.
\end{equation}
Since $u$ satisfies the integral equation
\begin{equation}\label{s2}
  u(x)=C\int_{\mathbb{R}^{n}}\frac{{u^{p}({y}})}{|x-y|^{n-2m}|y|^{a}}d{y},
\end{equation}
it follows that, for $x\neq0$,
\begin{eqnarray}\label{s3}
\bar{u}(x)&=&\frac{C}{|x|^{n-2m}}\int_{\mathbb{R}^{n}}\frac{{u^{p}({y}})}{\left|\frac{x}{|x|^{2}}-y\right|^{n-2m}|y|^{a}}d{y}
\\
\nonumber &=& \frac{C}{|x|^{n-2m}}\int_{\mathbb{R}^{n}}\frac{u^{p}\left(\frac{y}{|y|^{2}}\right)}{\left|\frac{x}{|x|^{2}}-\frac{y}{|y|^{2}}\right|^{n-2m}\frac{1}{|y|^{a}}}
\cdot\frac{1}{|y|^{2n}}d{y}
\\
\nonumber &=& \frac{C}{|x|^{n-2m}}\int_{\mathbb{R}^{n}}\frac{|x|^{n-2m}|y|^{n-2m}}{|x-y|^{n-2m}}\cdot\frac{u^{p}\left(\frac{y}{|y|^{2}}\right)}{|y|^{2n-a}}dy
\\
\nonumber &=& C\int_{\mathbb{R}^{n}}\frac{1}{|x-y|^{n-2m}}\cdot\frac{u^{p}\left(\frac{y}{|y|^{2}}\right)}{|y|^{n+2m-a}}dy\\
\nonumber &=& C\int_{\mathbb{R}^{n}}\frac{1}{|x-y|^{n-2m}}\cdot\frac{\overline{u}^{p}(y)}{|y|^{\tau}}dy,
\end{eqnarray}
where $\tau:=n+2m-a-p(n-2m)>0$.

We will apply the method of moving planes in integral forms to the integral equation \eqref{s3} and carry out the process of moving plane in the $x_{1}$ direction. For this purpose, we need some definitions.

Let $\lambda\leq0$ be an arbitrary non-positive real number and let the moving plane be
\begin{equation}\label{D1}
  T_{\lambda}:=\{x\in\mathbb{R}^{n}:\,x_{1}=\lambda\}.
\end{equation}
We denote
\begin{equation}\label{D2}
  \Sigma_{\lambda}:=\{x=(x_{1},x_{2},\cdots,x_{n})\in\mathbb{R}^{n}:\,x_{1}<\lambda\},
\end{equation}
and let
\begin{equation}\label{D3}
  x^{\lambda}:=(2\lambda-x_{1},x_{2},\cdots,x_{n})
\end{equation}
be the reflection of $x$ about the plane $T_{\lambda}$, and define
\begin{equation}\label{D4}
  \bar{u}_{\lambda}(x):=\bar{u}(x^{\lambda}), \,\,\,\,\,\,\,\, \omega_{\lambda}(x):=\bar{u}_{\lambda}(x)-\bar{u}(x).
\end{equation}

By properly exploiting some global properties of the integral equations, we will show that, for $\lambda$ sufficiently negative,
\begin{equation}\label{s4}
  \omega_{\lambda}(x)\geq0, \,\,\,\,\,\,\,\,\,\, \forall \,\, x\in\Sigma_{\lambda}\setminus\{0^{\lambda}\}.
\end{equation}
Then, we start moving the plane $T_{\lambda}$ from near $x_{1}=-\infty$ to the right as long as \eqref{s4} holds, until its limiting position and finally derive symmetry and monotonicity. Therefore, the moving plane process can be divided into two steps.

\emph{Step 1. Start moving the plane from near $x_{1}=-\infty$.} Define the set
\begin{equation}\label{s5}
  \Sigma^{-}_{\lambda}:=\{x\in\Sigma_{\lambda}\setminus\{0^{\lambda}\} \, | \, \omega_{\lambda}(x)<0\}.
\end{equation}
We can deduce from (2.3) that, for $x\in{\Sigma_{\lambda}}\setminus\{0^{\lambda}\}$,
\begin{eqnarray}\label{s6}
&&\omega_{\lambda}(x)=\bar{u}_{\lambda}(x)-\bar{u}(x)\\
\nonumber &=& C\int_{\mathbb{R}^{n}}\frac{1}{|x^{\lambda}-y|^{n-2m}}\frac{\bar{u}^{p}(y)}{|y|^{\tau}}dy-C\int_{\mathbb{R}^{n}}\frac{1}{|x-y|^{n-2m}}\frac{\bar{u}^{p}(y)}{|y|^{\tau}}dy
\\
\nonumber &=&C\int_{\Sigma_{\lambda}}\frac{1}{|x^{\lambda}-y|^{n-2m}}\frac{\bar{u}^{p}(y)}{|y|^{\tau}}dy+C\int_{\Sigma_{\lambda}}\frac{1}{|x^{\lambda}-y^{\lambda}|^{n-2m}}\frac{\bar{u}^{p}(y^{\lambda})}{|y^{\lambda}|^{\tau}}dy
\\
\nonumber & & -C\int_{\Sigma_{\lambda}}\frac{1}{|x-y|^{n-2m}}\frac{\bar{u}^{p}(y)}{|y|^{\tau}}dy-C\int_{\Sigma_{\lambda}}\frac{1}{|x-y^{\lambda}|^{n-2m}}\frac{\bar{u}^{p}(y^{\lambda})}{|y^{\lambda}|^{\tau}}dy
\\
\nonumber &=& C\int_{\Sigma_{\lambda}}\left(\frac{1}{|x-y|^{n-2m}}-\frac{1}{|x-y^{\lambda}|^{n-2m}}\right)\left(\frac{\bar{u}^{p}(y^{\lambda})}{|y^{\lambda}|^{\tau}}-\frac{\bar{u}^{p}(y)}{|y|^{\tau}}\right)dy\\
\nonumber &\geq& C\int_{\Sigma_{\lambda}}\left(\frac{1}{|x-y|^{n-2m}}-\frac{1}{|x-y^{\lambda}|^{n-2m}}\right)\frac{\bar{u}^{p}(y^{\lambda})-\bar{u}^{p}(y)}{|y|^{\tau}}dy\\
\nonumber &\geq& C\int_{\Sigma^{-}_{\lambda}}\frac{p\bar{u}^{p-1}(y)}{|x-y|^{n-2m}}\cdot\frac{\omega_{\lambda}(y)}{|y|^{\tau}}dy
\end{eqnarray}
In particular, for $x\in{\Sigma^{-}_{\lambda}}$, we have
\begin{equation}\label{s7}
  0>\omega_{\lambda}(x)\geq{C}\int_{\Sigma^{-}_{\lambda}}\frac{p\bar{u}^{p-1}(y)}{|x-y|^{n-2m}}\cdot\frac{\omega_{\lambda}(y)}{|y|^{\tau}}dy.
\end{equation}

By Hardy-Littlewood-Sobolev inequality, one gets, for arbitrary $\frac{n}{n-2m}<q<\infty$,
\begin{eqnarray}\label{s8}
\|\omega_{\lambda}\|_{{L^{q}(\Sigma^{-}_{\lambda})}}&\leq& C\left\|\int_{\Sigma^{-}_{\lambda}}\frac{p\bar{u}^{p-1}(y)}{|x-y|^{n-2m}}\cdot\frac{\omega_{\lambda}(y)}{|y|^{\tau}}dy\right\|_{{L^{q}(\Sigma^{-}_{\lambda})}}
\\
\nonumber &\leq&C\left\|\frac{\bar{u}^{p-1}(x)}{|x|^{\tau}}\cdot\omega_{\lambda}(x)\right\|_{{L^{\frac{nq}{n+2mq}}(\Sigma^{-}_{\lambda})}}
\\
\nonumber &\leq& C\left\|\frac{\bar{u}^{p-1}}{|x|^{\tau}}\right\|_{L^{\frac{n}{2m}}(\Sigma^{-}_{\lambda})}\cdot\|\omega_{\lambda}\|_{{L^{q}(\Sigma^{-}_{\lambda})}}
\end{eqnarray}
Since $\frac{u^{p-1}}{|x|^{a}}\in L^{\frac{n}{2m}}_{loc}(\mathbb{R}^{n})$, we have, for any $r>0$,
\begin{eqnarray}\label{s9}
\int_{|x|\geq{r}}\frac{\bar{u}^{(p-1)\frac{n}{2m}}(x)}{|x|^{\tau\frac{n}{2m}}}dx&=&
\int_{|x|\geq{r}}\frac{1}{|x|^{(\tau+(p-1)(n-2m))\frac{n}{2m}}}u^{(p-1)\frac{n}{2m}}\Big(\frac{x}{|x|^{2}}\Big)dx
\\
\nonumber &=&\int_{|x|\leq{\frac{1}{r}}}\frac{1}{|x|^{2n-(4m-a)\frac{n}{2m}}}u^{(p-1)\frac{n}{2m}}(x)dx
\\
\nonumber &=&\int_{|x|\leq{\frac{1}{r}}}\frac{u^{(p-1)\frac{n}{2m}}(x)}{|x|^{\frac{an}{2m}}}dx<+\infty.
\end{eqnarray}
Therefore, there exists a $\Lambda_{0}$ sufficiently large, such that, for any $\lambda\leq-\Lambda_{0}$,
\begin{equation}\label{s10}
  C\left\|\frac{\bar{u}^{p-1}(x)}{|x|^{\tau}}\right\|_{L^{\frac{n}{2m}}(\Sigma^{-}_{\lambda})}\leq\frac{1}{2}.
\end{equation}
Thus, we must have, for any $\frac{n}{n-2m}<q<\infty$,
\begin{equation}\label{s11}
\|\omega_{\lambda}\|_{{L^{q}(\Sigma^{-}_{\lambda})}}=0,
\end{equation}
Combining this with \eqref{s6} implies $\Sigma^{-}_{\lambda}=\emptyset$, and hence
\begin{equation}\label{s12}
\omega_{\lambda}(x)\geq0, \,\,\,\,\,\, \,\,\,\, \forall \, x\in\Sigma_{\lambda}\setminus\{0^{\lambda}\}.
\end{equation}

\emph{Step 2. Move the plane to the limiting position to derive symmetry and monotonicity.} Now we move the plane $T_{\lambda}$ to the right as long as \eqref{s4} holds. Define
\begin{equation}\label{s13}
\lambda_{0}:=\sup\{\lambda\in{\mathbb{R}}\, | \, \omega_{\rho}\geq0 \,\,\, \text{in} \,\,\, \Sigma_{\rho}\setminus\{0^{\rho}\}, \,\, \forall \, \rho\leq\lambda\}.
\end{equation}
By applying a entirely similar argument as in Step1, we can also start moving the plane from near $x_{1}=+\infty$ to the left, thus we must have $\lambda_{0}<+\infty$. Now, we will show that $\lambda_{0}=0$.

Suppose on the contrary that $\lambda_{0}<0$, we will show that
\begin{equation}\label{s14}
\omega_{\lambda_{0}}(x)\equiv0,\,\,\,\,\,\,\, \forall \,\, {x}\in\Sigma_{\lambda_{0}}\setminus\{0^{\lambda_{0}}\}.
\end{equation}

We prove \eqref{s14} by contradiction arguments. Suppose on the contrary that $\omega_{\lambda_{0}}(x)\geq0$, but $\omega_{\lambda_{0}}(x)$ is not identically zero in $\Sigma_{\lambda_{0}}\setminus\{0^{\lambda_{0}}\}$. We will obtain a contradiction with \eqref{s13} via showing that the plane $T_{\lambda}$ can be moved a little bit further to the right, more precisely, there exist an $0<\varepsilon<|\lambda_{0}|$ small enough, such that $w_{\lambda}\geq0$ in $\Sigma_{\lambda}\setminus\{0^{\lambda}\}$ for all $\lambda\in[\lambda_{0},\lambda_{0}+\varepsilon)$.

It can be clearly seen from \eqref{s8} and \eqref{s10} in \emph{Step 1} that, our goal is to prove that, one can choose $\varepsilon>0$ sufficiently small such that, for all $\lambda\in[\lambda_{0},\lambda_{0}+\varepsilon)$,
\begin{equation}\label{s15}
\left\|\frac{\bar{u}^{p-1}(x)}{|x|^{\tau}}\right\|_{L^{\frac{n}{2m}}(\Sigma^{-}_{\lambda})}\leq\frac{1}{2C},
\end{equation}
where the constant $C$ is the same as in \eqref{s8} and \eqref{s10}.

In fact, by \eqref{s9}, we can choose $R>0$ large enough, such that
\begin{equation}\label{s16}
\left(\int_{|x|\geq{R}}\frac{\bar{u}^{(p-1)\frac{n}{2m}}(x)}{|x|^{\tau\frac{n}{2m}}}dx\right)^{\frac{2m}{n}}<\frac{1}{4C}.
\end{equation}
Now fix this $R$, in order to derive \eqref{s15}, we only need to show
\begin{equation}\label{s17}
\lim_{\lambda\rightarrow\lambda_{0}+}\mu\left({\Sigma^{-}_{\lambda}}\cap{B_{R}(0)}\right)=0.
\end{equation}
To this end, we define $E_{\delta}:=\{x\in(\Sigma_{\lambda_{0}}\setminus\{0^{\lambda_{0}}\})\cap{B_{R}(0)}\,|\,w_{\lambda_{0}}(x)>\delta\}$
and $F_{\delta}:=({\Sigma_{\lambda_{0}}}\cap{B_{R}(0)})\setminus{E_{\delta}}$ for any $\delta>0$, and let $D_{\lambda}:=({\Sigma_{\lambda}}\setminus{\Sigma_{\lambda_{0}}})\cap{B_{R}(0)}$ for any $\lambda>\lambda_{0}$. Then, one can easily verify that
\begin{equation}\label{s18}
\lim_{\delta\rightarrow{0}^{+}}\mu(F_{\delta})=0, \,\,\,\,\,\,\, \lim_{\lambda\rightarrow\lambda_{0}+}\mu(D_{\lambda})=0,
\end{equation}
\begin{equation}\label{s19}
{\Sigma^{-}_{\lambda}}\cap{B_{R}(0)}={\Sigma^{-}_{\lambda}}\cap(E_{\delta}\cup{F_{\delta}}\cup{D_{\lambda}})
\subset({\Sigma^{-}_{\lambda}}\cap E_{\delta})\cup{F_{\delta}}\cup{D_{\lambda}}.
\end{equation}
For an arbitrary fixed $\eta>0$, one can choose a $\delta>0$ small enough, such that $\mu(F_{\delta})\leq\eta$. For this fixed $\delta$, we are to prove
\begin{equation}\label{s20}
\lim_{\lambda\rightarrow\lambda_{0}+}\mu(\Sigma_{\lambda}^{-}\cap{E_{\delta}})=0.
\end{equation}
Indeed, one can observe that $\bar{u}(x^{\lambda_{0}})-\bar{u}(x^{\lambda})=\omega_{\lambda_{0}}(x)-\omega_{\lambda}(x)>\delta$
for all $x\in{\Sigma_{\lambda}^{-}\cap{E_{\delta}}}$. It follows that
\begin{equation}\label{s21}
  (\Sigma_{\lambda}^{-}\cap{E_{\delta}})\subset{G_{\delta}^{\lambda}}
  :=\{x\in{B_{R}(0)}\cap\big((\Sigma_{\lambda_{0}}\setminus\{0^{\lambda_{0}}\})\setminus\{0^{\lambda}\}\big)\,|\,\bar{u}(x^{\lambda_{0}})-\bar{u}(x^{\lambda})>\delta\}.
\end{equation}
By Chebyshev's inequality, we get
\begin{eqnarray}\label{s22}
\mu(G_{\delta}^{\lambda})&\leq&\frac{1}{\delta^{r}}\int_{G_{\delta}^{\lambda}}\left|\bar{u}(x^{\lambda_{0}})-\bar{u}(x^{\lambda})\right|^{r}dx \\
\nonumber &\leq&\frac{1}{\delta^{r}}\int_{B_{R}(0)}\left|\bar{u}(x)-\bar{u}\left(x+2(\lambda-\lambda_{0})e_{1}\right)\right|^{r}dx
\end{eqnarray}
for any $1\leq{r}<\frac{n}{n-2m}$, where $e_{1}=(1,0......,0)\in{\mathbb{R}^{n}}$, and hence
\begin{equation}\label{s23}
\lim_{\lambda\rightarrow\lambda_{0}+}\mu(G^{\lambda}_{\delta})=0,
\end{equation}
from which \eqref{s20} follows immediately.

Therefore, by \eqref{s18}, \eqref{s19} and \eqref{s20}, we have
\begin{equation}\label{s24}
\lim_{\lambda\rightarrow\lambda_{0}+}\mu(\Sigma_{\lambda}^{-}\cap{B_{R}(0)})\leq\mu(F_{\delta})\leq\eta.
\end{equation}
Since $\eta>0$ is arbitrarily chosen, \eqref{s17} follows immediately from \eqref{s24}. Combining \eqref{s16} and \eqref{s17}, we finally arrive at \eqref{s15}.

From the last inequality of \eqref{s8}, we have, for any $\frac{n}{n-2m}<q<\infty$,
\begin{equation}\label{s25}
\|\omega_{\lambda}\|_{L^{q}(\Sigma^{-}_{\lambda})}\leq C\left\|\frac{\bar{u}^{p-1}}{|x|^{\tau}}\right\|_{L^{\frac{n}{2m}}(\Sigma^{-}_{\lambda})}
\cdot\|\omega_{\lambda}\|_{L^{q}(\Sigma^{-}_{\lambda})}.
\end{equation}
By \eqref{s15} and the above estimate, we deduce that, there exists an $\varepsilon>0$ sufficiently small, such that, for all $\lambda\in[\lambda_{0},\lambda_{0}+\varepsilon)$, $\|\omega_{\lambda}\|_{{L^{q}(\Sigma^{-}_{\lambda})}}=0$, thus $\mu(\Sigma^{-}_{\lambda})=0$. Furthermore, by \eqref{s6}, we have $\Sigma^{-}_{\lambda}=\emptyset$, and hence $\omega_{\lambda}(x)\geq0$ in $\Sigma_{\lambda}\setminus\{0^{\lambda}\}$ for all $\lambda\in[\lambda_{0},\lambda_{0}+\varepsilon)$. This contradicts with the definition of ${\lambda_{0}}$. Therefore, \eqref{s14} must hold. By \eqref{s6} and \eqref{s14}, we get, for any $x\in{\Sigma_{\lambda_{0}}}$,
\begin{eqnarray}\label{s26}
\nonumber 0=\omega_{\lambda_{0}}(x)&=&\bar{u}(x^{\lambda_{0}})-\bar{u}(x)
\\
 &=&C\int_{\Sigma_{\lambda_{0}}}\left(\frac{1}{|x-y|^{n-2m}}-\frac{1}{|x-y^{\lambda_{0}}|^{n-2m}}\right)
 \left(\frac{1}{|y^{\lambda_{0}}|^{\tau}}-\frac{1}{|y|^{\tau}}\right)\bar{u}^{p}(y)dy
\\
\nonumber &>&0.
\end{eqnarray}
That is a contradiction! Thus we must have $\lambda_{0}=0$, and hence
\begin{equation}\label{s27}
u(-x_{1},x_{2},....,x_{n})\geq{u(x_{1},x_{2},....,x_{n})},\,\,\,\,\, \forall \,\, {x}\in{\Sigma_{0}}.
\end{equation}
We can also move the plane from $x_{1}=+\infty$ to the left and the limiting position is also $\lambda_{0}=0$, so one has
\begin{equation}\label{s28}
u(-x_{1},x_{2},....,x_{n})\leq{u(x_{1},x_{2},....,x_{n})},\,\,\, \,\,\,\,\, \forall \,\, {x}\in{\Sigma_{0}}.
\end{equation}
Therefore,
\begin{equation}\label{s29}
u(-x_{1},x_{2},....,x_{n})\equiv{u(x_{1},x_{2},....,x_{n})},\,\,\, \,\,\,\,\, \forall \,\, {x}\in{R^{n}},
\end{equation}
that is, $u(x)$ is symmetric with respect to the plane $T_{0}=\{x\in\mathbb{R}^{n}\,|\,x_{1}=0\}$.

Since the equation is invariant under rotation, the $x_{1}$ direction can be chosen arbitrarily. We conclude that the positive solution $u(x)$ must be radially symmetric and monotone decreasing about the origin $0\in{\mathbb{R}^{n}}$. This finishes our proof of Theorem \ref{symmetry}.
\end{proof}

\subsection{Pohozaev identity and nonexistence of positive radially symmetric solutions}
By Theorem \ref{symmetry}, we deduce that the positive classical solution $u$ to PDE \eqref{PDE} is a positive radially symmetric solution to IE \eqref{IE}, i.e., $u(x)=u(|x|)>0$. Next, we will show that there is no positive radially symmetric classical solutions to \eqref{IE}, which leads to a contradiction.
\begin{thm}\label{nonexistence}
Assume $n\geq3$, $1\leq m<\frac{n}{2}$, $0\leq a<2m$ and $1<p<\frac{n+2m-2a}{n-2m}$, then \eqref{IE} has no positive radially symmetric classical solutions.
\end{thm}
\begin{proof}
Suppose $u(x)=u(|x|)>0$ is a positive radially symmetric classical solution to \eqref{IE}, that is,
\begin{equation}\label{P1}
  u(x)=C\int_{\mathbb{R}^{n}}\frac{{u^{p}({y}})}{|x-y|^{n-2m}|y|^{a}}d{y}.
\end{equation}
Then, for any $\mu>0$,
\begin{equation}\label{P2}
  u({\mu}x)=C\int_{\mathbb{R}^{n}}\frac{{u^{p}({y}})}{|{\mu}x-y|^{n-2m}|y|^{a}}d{y}.
\end{equation}
Take the derivatives on both sides of \eqref{P2} with respect to ${\mu}$ and let ${\mu}=1$, we get
\begin{eqnarray}\label{P3}
x\cdot\nabla{u(x)}&=& C\frac{\mathrm{d} }{\mathrm{d}\mu}\int_{\mathbb{R}^{n}}\frac{{u^{p}({y}})}{|{\mu}x-y|^{n-2m}|y|^{a}}d{y}\bigg|_{{\mu}=1}
\\
\nonumber &=& -(n-2m)C\int_{\mathbb{R}^{n}}\frac{{u^{p}({y}})({\mu}x-y)\cdot x}{|{\mu}x-y|^{n-2m-2}|y|^{a}}d{y}\bigg|_{{\mu}=1}
\\
\nonumber &=& -(n-2m)C\int_{\mathbb{R}^{n}}\frac{{u^{p}({y}})(x-y)\cdot x}{|x-y|^{n-2m-2}|y|^{a}}d{y}.
\end{eqnarray}

Multiply both sides of \eqref{P3} by $\frac{{u^{p}({x}})}{|x|^{a}}$ and integrate on $B_{r}(0)$ for any $r>0$, one has
\begin{eqnarray}\label{P4}
LHS &=& \int_{B_{r}(0)}(x\cdot\nabla{u(x)}){\frac{{u^{p}({x}})}{|x|^{a}}}dx
\\
\nonumber &=& \int_{0}^{r}\int_{{\partial}B_{s}(0)}s\frac{\mathrm{d}(u(s)) }{\mathrm{d} s}\cdot{\frac{{u^{p}({s}})}{s^{a}}}d{\sigma}ds
\\
\nonumber &=& \int_{0}^{r}\frac{w_{n}{s^{n-a}}}{p+1}d\big(u^{p+1}(s)\big)
\\
\nonumber &=& \frac{w_{n}}{p+1}r^{n-a}u^{p+1}(r)-\frac{(n-a)w_{n}}{p+1}\int_{0}^{r}u^{p+1}(s)s^{n-a-1}ds
\\
\nonumber &=& \frac{r^{1-a}}{p+1}\int_{{\partial}B_{r}(0)}u^{p+1}(x)d\sigma-\frac{n-a}{p+1}\int_{B_{r}(0)}\frac{u^{p+1}(x)}{|x|^{a}}dx,
\end{eqnarray}
and
\begin{equation}\label{P5}
  RHS=-(n-2m)C\int_{B_{r}(0)}\frac{u^{p}(x)}{|x|^{a}}\int_{\mathbb{R}^{n}}\frac{(x-y)\cdot x{u}^{p}(y)}{|x-y|^{n-2m-2}|y|^{a}}dy,
\end{equation}
where $w_{n}$ denotes the area of the unit sphere. Since $u$ is a positive radially symmetric classical solution to \eqref{IE}, we have $u(r)$ monotone decreases about $r\geq0$ and
\begin{eqnarray}\label{P6}
u(x)=u(r)&=& C\int_{\mathbb{R}^{n}}\frac{u^{p}(y)}{|x-y|^{n-2m}|y|^{a}}dy
\\
\nonumber &=& \frac{C}{r^{n-2m}}\int_{\mathbb{R}^{n}}\frac{u^{p}(y)}{\left|\frac{x}{|x|}-\frac{y}{r}\right|^{n-2m}|y|^{a}}dy
\\
\nonumber &\geq& \frac{C}{r^{n-2m}}\int_{0}^{r}\int_{{\partial}B_{s}(0)}\frac{u^{p}(s)}{\left|\frac{x}{|x|}-\frac{s\sigma}{r}\right|^{n-2m}s^{a}}d(s\sigma)ds,
\end{eqnarray}
where $r=|x|$ and $\sigma$ is an arbitrary unit vector on $\partial B_{1}(0)$. Observe that
\begin{equation}\label{P7}
  \frac{1}{\left|\frac{x}{|x|}-\frac{s\sigma}{r}\right|^{n-2m}}\geq\frac{1}{2^{n-2m}},\,\,\,\,\,\,\,\,\,\,\, \forall \,\,s\in[0,r] \,\,\,\text{and} \,\,\, \forall \,\, \sigma\in\partial B_{1}(0),
\end{equation}
thus we infer from \eqref{P6} that
\begin{eqnarray}\label{P8}
\nonumber
u(x)=u(r)
 \nonumber &\geq& \frac{C}{r^{n-2m}}\int_{0}^{r}w_{n}s^{n-1-a}\frac{u^{p}(s)}{2^{n-2m}}ds
\\
\nonumber &\geq& \frac{Cw_{n}}{{(2r)}^{n-2m}}u^{p}(r)\int_{0}^{r}s^{n-1-a}ds
\\
\nonumber &=& \frac{Cw_{n}}{{2}^{n-2m}(n-a)}\cdot\frac{u^{p}(r)}{r^{n-2m}}r^{n-a}
\\
\nonumber &=& \frac{Cw_{n}r^{2m-a}u^{p}(r)}{{2}^{n-2m}(n-a)}.
\end{eqnarray}
Therefore,
\begin{equation}\label{P9}
  u^{p-1}(r)\leq\frac{{2}^{n-2m}(n-a)}{Cw_{n}r^{2m-a}}.
\end{equation}
Let $\tilde{C}=\Big({\frac{{2}^{n-2m}(n-a)}{Cw_{n}}}\Big)^{\frac{1}{p-1}}>0$, then one has the following decay estimate
\begin{equation}\label{P10}
  u(r)\leq\frac{\tilde{C}}{r^{\frac{2m-a}{p-1}}}, \quad\quad\, \forall \,\, r>0.
\end{equation}
Note that $1<p<\frac{n+2m-2a}{n-2m}$, we derive from the decay estimate \eqref{P10} that
\begin{equation}\label{P11}
\int_{\mathbb{R}^{n}}\frac{u^{p+1}(x)}{|x|^{a}}dx<\infty,
\end{equation}
and hence
\begin{equation}\label{P12}
 \int_{0}^{\infty}r^{-a}\left(\int_{{\partial}B_{r}(0)}u^{p+1}(x)d{\sigma}\right)dr<\infty.
\end{equation}
Thus there exists a sequence $\{r_{j}\}$, such that ${r_{j}}\rightarrow+\infty$ as $j\rightarrow\infty$ and
\begin{equation}\label{P13}
r_{j}^{1-a}\left(\int_{{\partial}B_{r_{j}}(0)}u^{p+1}(x)d{\sigma}\right)\rightarrow0,\,\,\,\,\,\quad \text{as} \,\, j\rightarrow\infty.
\end{equation}
By letting $r=r_{j}\rightarrow+\infty$ in \eqref{P4} and \eqref{P5}, we conclude from \eqref{P11} and \eqref{P13} that
\begin{equation}\label{P14}
-\frac{n-a}{p+1}\int_{\mathbb{R}^{n}}\frac{u^{p+1}(x)}{|x|^{a}}dx
=-(n-2m)C\int_{\mathbb{R}^{n}}\frac{u^{p}(x)}{|x|^{a}}\int_{\mathbb{R}^{n}}\frac{{u^{p}({y}})x\cdot(x-y)}{|x-y|^{n-2m+2}|y|^{a}}d{y}dx.
\end{equation}
At the same time, by direct calculations, we have
\begin{eqnarray}\label{P15}
&& \frac{2m-n}{2}\int_{\mathbb{R}^{n}}\frac{u^{p+1}(x)}{|x|^{a}}dx \\
\nonumber &=&\frac{2m-n}{2}\int_{\mathbb{R}^{n}}\frac{u^{p}(x)}{|x|^{a}}C\int_{\mathbb{R}^{n}}\frac{{u^{p}({y}})}{|x-y|^{n-2m}|y|^{a}}d{y}dx
\\
\nonumber &=& \frac{(2m-n)C}{2}\int_{\mathbb{R}^{n}}\int_{\mathbb{R}^{n}}\frac{|x-y|^{2}u^{p}(x)u^{p}(y)}{|x|^{a}|y|^{a}|x-y|^{n-2m+2}}dydx
\\
\nonumber &=& \frac{(2m-n)C}{2}\int_{\mathbb{R}^{n}}\int_{\mathbb{R}^{n}}\frac{[(x-y)\cdot x+(y-x)\cdot y]u^{p}(x)u^{p}(y)}{|x|^{a}|y|^{a}|x-y|^{n-2m+2}}dydx
\\
\nonumber &=& -(n-2m)C\int_{\mathbb{R}^{n}}\frac{u^{p}(x)}{|x|^{a}}\int_{\mathbb{R}^{n}}\frac{{u^{p}({y}})x\cdot (x-y)}{|x-y|^{n-2m+2}|y|^{a}}d{y}dx.
\end{eqnarray}
Combining \eqref{P14} and \eqref{P15}, we deduce further that
\begin{equation}\label{P16}
\left(\frac{2m-n}{2}+\frac{n-a}{p+1}\right)\int_{\mathbb{R}^{n}}\frac{u^{p+1}(x)}{|x|^{a}}dx=0.
\end{equation}
Since $1<p<\frac{n+2m-2a}{n-2m}$, it is easy to see that
\begin{equation}\label{P17}
\frac{2m-n}{2}+\frac{n-a}{p+1}>0,
\end{equation}
thus we must have
\begin{equation}\label{P18}
\int_{\mathbb{R}^{n}}\frac{u^{p+1}(x)}{|x|^{a}}dx=0,
\end{equation}
which is a contradiction with $u>0$! Therefore, \eqref{IE} does not have any positive radially symmetric classical solutions.
\end{proof}

Since we have proved the positive classical solution $u$ to PDE \eqref{PDE} is also a positive radially symmetric solution to IE \eqref{IE}, Theorem \ref{nonexistence} leads to a contradiction. Therefore, we must have $u\equiv0$ in $\mathbb{R}^{n}$, that is, the unique nonnegative solution to PDE \eqref{PDE} is $u\equiv0$ in $\mathbb{R}^{n}$.

This concludes the proof of Theorem \ref{Thm0}.

\section{Proof of Theorem \ref{Thm1}}
In this section, we will prove Theorem \ref{Thm1} via the method of moving planes in local way and blowing-up techniques.

\subsection{Boundary layer estimates}
In this subsection, we will first establish the following boundary layer estimates by applying Kelvin transform and the method of moving planes in local way. The properties of the boundary $\partial\Omega$ will play a crucial role in our discussions.
\begin{thm}\label{Boundary}
Assume one of the following two assumptions
\begin{equation*}
  \text{i)} \,\,\, \Omega \,\, \text{is strictly convex}, \,\, 1<p<\frac{n+2m}{n-2m}, \quad\quad\quad\, \text{or} \quad\quad\quad\, \text{ii)} \,\,\, 1<p\leq\frac{n+2}{n-2}
\end{equation*}
holds. Then, there exists a $\bar{\delta}>0$ depending only on $\Omega$ such that, for any positive solution $u\in C^{2m}(\Omega)\cap C^{2m-2}(\overline{\Omega})$ to the higher order Navier problem \eqref{tNavier}, we have
\begin{equation*}
  \|u\|_{L^{\infty}(\overline{\Omega}_{\bar{\delta}})}\leq C(n,m,p,\lambda_{1},\Omega),
\end{equation*}
where the boundary layer $\overline{\Omega}_{\bar{\delta}}:=\{x\in\overline{\Omega}\,|\,dist(x,\partial\Omega)\leq\bar{\delta}\}$.
\end{thm}
\begin{rem}\label{remark2}
When $m=1$, Theorem \ref{Boundary} still holds for $p=\frac{n+2}{n-2}$.
\end{rem}
\begin{proof}
We will carry out our proof of Theorem \ref{Boundary} by discussing the two different assumptions i) and ii) separately.

\emph{Case i)} $\Omega$ is strictly convex and $1<p<\frac{n+2m}{n-2m}$. For any $x^{0}\in\partial\Omega$, let $\nu^{0}$ be the unit internal normal vector of $\partial\Omega$ at $x^{0}$, we will show that $u(x)$ is monotone increasing along the internal normal direction in the region
\begin{equation}\label{3-1}
  \overline{\Sigma_{\delta_{0}}}=\left\{x\in\overline{\Omega}\,|\,0\leq(x-x^{0})\cdot\nu^{0}\leq\delta_{0}\right\},
\end{equation}
where $\delta_{0}>0$ depends only on $x^{0}$ and $\Omega$.

To this end, we define the moving plane by
\begin{equation}\label{3-2}
  T_{\lambda}:=\{x\in\mathbb{R}^n\,|\,(x-x^{0})\cdot\nu^{0}=\lambda\},
\end{equation}
and denote
\begin{equation}\label{3-2'}
  \Sigma_{\lambda}:=\{x\in\Omega\,|\,0<(x-x^{0})\cdot\nu^{0}<\lambda\}
\end{equation}
for $\lambda>0$, and let $x^{\lambda}$ be the reflection of the point $x$ about the plane $T_{\lambda}$.

Let $u_{i}:=(-\Delta)^{i}u$ for $1\leq i\leq m-1$. By maximum principle, we have
\begin{equation}\label{3-3}
  u_{i}(x)>0 \quad\quad\, \text{in} \,\, \Omega
\end{equation}
for $1\leq i\leq m-1$. Define
\begin{equation}\label{3-4}
  W^{\lambda}(x):=u(x^{\lambda})-u(x) \quad\quad\, \text{and} \quad\quad\, W^{\lambda}_{i}(x):=u_{i}(x^{\lambda})-u_{i}(x)
\end{equation}
for $1\leq i\leq m-1$. Then we can deduce from \eqref{tNavier} that, for any $\lambda$ satisfying the reflection of $\Sigma_{\lambda}$ is contained in $\Omega$,
\begin{equation}\label{3-5}
\left\{{\begin{array}{l} {-\Delta W^{\lambda}_{m-1}(x)=u^{p}(x^{\lambda})-u^{p}(x)=p\eta^{p-1}_{\lambda}(x)W^{\lambda}(x), \,\,\,\,\,\, x\in\Sigma_{\lambda},}\\  {} \\ {-\Delta W^{\lambda}_{m-2}(x)=W^{\lambda}_{m-1}(x), \,\,\,\,\,\, x\in\Sigma_{\lambda},} \\ \cdots\cdots \\ {-\Delta W^{\lambda}(x)=W^{\lambda}_1(x), \,\,\,\,\,\, x\in\Sigma_{\lambda},} \\ {} \\
{W^{\lambda}(x)\geq0, \, W^{\lambda}_{1}(x)\geq0, \cdots, W^{\lambda}_{m-1}(x)\geq0, \,\,\,\,\,\, x\in\partial\Sigma_{\lambda},} \\ \end{array}}\right.
\end{equation}
where $\eta_{\lambda}(x)$ is valued between $u(x^{\lambda})$ and $u(x)$ by mean value theorem. Now, we will prove that there exists some $\delta>0$ sufficiently small (depending on $m$, $p$, $\|u\|_{L^{\infty}(\overline{\Omega})}$ and $\Omega$), such that
\begin{equation}\label{3-6}
  W^{\lambda}(x)\geq 0 \quad\quad\, \text{in} \,\, \Sigma_{\lambda}
\end{equation}
for all $0<\lambda\leq\delta$. This provides a starting point to move the plane $T_{\lambda}$.

Indeed, suppose on the contrary that there exists a $0<\lambda\leq\delta$ such that
\begin{equation}\label{3-7}
  W^{\lambda}(x)<0 \quad\quad\, \text{somewhere in} \,\, \Sigma_{\lambda}.
\end{equation}
Let
\begin{equation}\label{3-8}
  \zeta(x):=\cos\frac{(x-x^{0})\cdot\nu^{0}}{\delta},
\end{equation}
then it follows that $\zeta(x)\in[\cos1,1]$ for any $x\in\Sigma_{\lambda}$ and $-\frac{\Delta\zeta(x)}{\zeta(x)}=\frac{1}{\delta^2}$. Define
\begin{equation}\label{3-9}
  \overline{W^{\lambda}}(x):=\frac{W^{\lambda}(x)}{\zeta(x)} \quad\quad \text{and} \quad\quad \overline{W^{\lambda}_{i}}(x):=\frac{W^{\lambda}_{i}(x)}{\zeta(x)}
\end{equation}
for $i=1,\cdots, m-1$ and $x\in\Sigma_{\lambda}$. Then there exists a $x_{0}\in\Sigma_{\lambda}$ such that
\begin{equation}\label{3-10}
  \overline{W^{\lambda}}(x_{0})=\min_{\overline{\Sigma_{\lambda}}}\overline{W^{\lambda}}(x)<0.
\end{equation}
Since
\begin{equation}\label{3-11}
  -\Delta W^{\lambda}(x_{0})=-\Delta\overline{W^{\lambda}}(x_{0})\zeta(x_{0})-2\nabla\overline{W^{\lambda}}(x_{0})\cdot\nabla\zeta(x_{0})
  -\overline{W^{\lambda}}(x_{0})\Delta\zeta(x_{0}),
\end{equation}
one immediately has
\begin{equation}\label{3-12}
  W^{\lambda}_{1}(x_{0})=-\Delta W^{\lambda}(x_{0})\leq\frac{1}{\delta^2}W^{\lambda}(x_{0})<0.
\end{equation}
Thus there exists a $x_{1}\in\Sigma_{\lambda}$ such that
\begin{equation}\label{3-13}
  \overline{W^{\lambda}_{1}}(x_{1})=\min_{\overline{\Sigma_{\lambda}}}\overline{W^{\lambda}_{1}}(x)<0.
\end{equation}
Similarly, it follows that
\begin{equation}\label{3-14}
  W^{\lambda}_{2}(x_{1})=-\Delta W^{\lambda}_{1}(x_{1})\leq\frac{1}{\delta^2}W^{\lambda}_{1}(x_{1})<0.
\end{equation}
Continuing this way, we get $\{x_{i}\}_{i=1}^{ m-1}\subset\Sigma_{\lambda}$ such that
\begin{equation}\label{3-15}
  \overline{W^{\lambda}_{i}}(x_{i})=\min_{\overline{\Sigma_{\lambda}}}\overline{W^{\lambda}_{i}}(x)<0,
\end{equation}
\begin{equation}\label{3-16}
  W^{\lambda}_{i+1}(x_{i})=-\Delta W^{\lambda}_{i}(x_{i})\leq\frac{1}{\delta^2}W^{\lambda}_{i}(x_{i})<0
\end{equation}
for $i=1,2,\cdots, m-2$, and
\begin{equation}\label{3-17}
  \overline{W^{\lambda}_{ m-1}}(x_{ m-1})=\min_{\overline{\Sigma_{\lambda}}}\overline{W^{\lambda}_{ m-1}}(x)<0,
\end{equation}
\begin{equation}\label{3-18}
  p\eta^{p-1}_{\lambda}(x_{ m-1})W^{\lambda}(x_{ m-1})=-\Delta W^{\lambda}_{ m-1}(x_{ m-1})\leq\frac{1}{\delta^2}W^{\lambda}_{ m-1}(x_{ m-1})<0.
\end{equation}
Therefore, we have
\begin{eqnarray}\label{3-19}
  W^{\lambda}(x_{0}) &\geq& \delta^2W^{\lambda}_{1}(x_{0})\geq \delta^{2}W^{\lambda}_{1}(x_{1})\frac{\zeta(x_{0})}{\zeta(x_{1})}
  \geq\delta^{4}W^{\lambda}_{2}(x_{1})\frac{\zeta(x_{0})}{\zeta(x_{1})} \\
 \nonumber &\geq& \delta^{4}W^{\lambda}_{2}(x_{2})\frac{\zeta(x_{0})}{\zeta(x_{2})}\geq\delta^{6}W^{\lambda}_{3}(x_{2})\frac{\zeta(x_{0})}{\zeta(x_{2})}
 \geq\delta^{6}W^{\lambda}_{3}(x_{3})\frac{\zeta(x_{0})}{\zeta(x_{3})} \\
 \nonumber  &\geq& \cdots\cdots\geq\delta^{2m-2}W^{\lambda}_{ m-1}(x_{ m-1})\frac{\zeta(x_{0})}{\zeta(x_{ m-1})} \\
 \nonumber  &\geq& p\delta^{2m}\eta^{p-1}_{\lambda}(x_{ m-1})W^{\lambda}(x_{ m-1})\frac{\zeta(x_{0})}{\zeta(x_{ m-1})} \\
 \nonumber  &\geq& p\delta^{2m}\|u\|^{p-1}_{L^{\infty}(\overline{\Omega})}W^{\lambda}(x_{0}),
\end{eqnarray}
that is,
\begin{equation}\label{3-20}
  1\leq p\delta^{2m}\|u\|^{p-1}_{L^{\infty}(\overline{\Omega})},
\end{equation}
which is absurd if we choose $\delta>0$ small enough such that
\begin{equation}\label{3-21}
  0<\delta<\left(p\|u\|^{p-1}_{L^{\infty}(\overline{\Omega})}\right)^{-\frac{1}{2m}}.
\end{equation}
So far, our conclusion is: the method of moving planes can be carried on up to $\lambda=\delta$.

Next, we will move the plane $T_{\lambda}$ further along the internal normal direction at $x^{0}$ as long as the property
\begin{equation}\label{3-22}
  W^{\lambda}(x)\geq0 \quad\quad\, \text{in} \,\, \Sigma_{\lambda}
\end{equation}
holds. One can conclude that the moving planes process can be carried on (with the property \eqref{3-22}) as long as the reflection of $\overline{\Sigma_{\lambda}}$ is still contained in $\Omega$.

In fact, let $T_{\lambda_{0}}$ be a plane such that \eqref{3-22} holds and the reflection of $\overline{\Sigma_{\lambda_{0}}}$ about $T_{\lambda_{0}}$ is contained in $\Omega$. Then there exists a $\kappa>0$ such that, the reflection of $\overline{\Sigma_{\lambda_{0}+\kappa}}$ about $T_{\lambda_{0}+\kappa}$ is still contained in $\Omega$. By \eqref{3-5}, \eqref{3-22} and strong maximum principles, one actually has
\begin{equation}\label{3-23}
  W^{\lambda_{0}}(x)>0, \quad\quad\, W_{i}^{\lambda_{0}}(x)>0 \quad\quad \text{in} \,\, \Sigma_{\lambda_{0}},
\end{equation}
thus there exists a constant $c_{\delta}>0$ such that
\begin{equation}\label{3-24}
  W^{\lambda_{0}}(x)\geq c_{\delta}>0, \quad\quad\, W_{i}^{\lambda_{0}}(x)\geq c_{\delta}>0 \quad\quad \text{in} \,\, \overline{\Sigma_{\lambda_{0}-\frac{\delta}{2}}}.
\end{equation}
By the continuity of $u$, we infer that, there exists a $0<\epsilon<\min\{\kappa,\frac{\delta}{2}\}$ such that, for any $\lambda\in(\lambda_{0},\lambda_{0}+\epsilon]$,
\begin{equation}\label{3-25}
  W^{\lambda}(x)>0, \quad\quad\, W_{i}^{\lambda}(x)>0 \quad\quad \text{in} \,\, \overline{\Sigma_{\lambda_{0}-\frac{\delta}{2}}}.
\end{equation}
Suppose there exists a $\lambda_{0}<\lambda\leq\lambda_{0}+\epsilon$ such that
\begin{equation}\label{3-26}
  W^{\lambda}(x)<0 \quad\quad\, \text{somewhere in} \,\, \Sigma_{\lambda}\setminus\overline{\Sigma_{\lambda_{0}-\frac{\delta}{2}}}.
\end{equation}
Let
\begin{equation}\label{3-27}
  \overline{\zeta}(x):=\cos\frac{\left(x-x^{0}-(\lambda_{0}-\frac{\delta}{2})\nu^{0}\right)\cdot\nu^{0}}{\delta} \quad\quad \text{and} \quad\quad \widetilde{W^{\lambda}}(x):=\frac{W^{\lambda}(x)}{\overline{\zeta}(x)}
\end{equation}
for $x\in\Sigma_{\lambda}\setminus\overline{\Sigma_{\lambda_{0}-\frac{\delta}{2}}}$. Then there exists a $x_{0}\in\Sigma_{\lambda}\setminus\overline{\Sigma_{\lambda_{0}-\frac{\delta}{2}}}$ such that
\begin{equation}\label{3-28}
  \widetilde{W^{\lambda}}(x_{0})=\min_{\overline{\Sigma_{\lambda}\setminus\overline{\Sigma_{\lambda_{0}-\frac{\delta}{2}}}}}\widetilde{W^{\lambda}}(x)<0,
\end{equation}
by using similar arguments as proving \eqref{3-19}, one can also arrive at
\begin{equation}\label{3-29}
  W^{\lambda}(x_{0})\geq p\delta^{2m}\|u\|^{p-1}_{L^{\infty}(\overline{\Omega})}W^{\lambda}(x_{0}),
\end{equation}
which contradicts with the choice of $\delta$. Therefore, we have proved that
\begin{equation}\label{3-30}
  W^{\lambda}(x)\geq0 \quad\quad\, \text{in} \,\, \Sigma_{\lambda}
\end{equation}
for any $\lambda\in(\lambda_{0},\lambda_{0}+\epsilon]$, that is, the plane $T_{\lambda}$ can be moved forward a little bit from $T_{\lambda_{0}}$.

Therefore, there exists a $\delta_{0}>0$ depending only on $x^{0}$ and $\Omega$ such that, $u(x)$ is monotone increasing along the internal normal direction in the region
\begin{equation}\label{3-31}
  \overline{\Sigma_{\delta_{0}}}:=\left\{x\in\overline{\Omega}\,|\,0\leq(x-x^{0})\cdot\nu^{0}\leq\delta_{0}\right\}.
\end{equation}
Since $\partial\Omega$ is $C^{2m-2}$, there exists a small $0<r_{0}<\frac{\delta_{0}}{8}$ depending on $x^{0}$ and $\Omega$ such that, for any $x\in B_{r_{0}}(x^{0})\cap\partial\Omega$, $u(x)$ is monotone increasing along the internal normal direction at $x$ in the region
\begin{equation}\label{3-32}
  \overline{\Sigma_{x}}:=\left\{z\in\overline{\Omega}\,\Big|\,0\leq(z-x)\cdot\nu_{x}\leq\frac{3}{4}\delta_{0}\right\}.
\end{equation}
where $\nu_{x}$ denotes the unit internal normal vector at the point $x$ ($\nu_{x^{0}}:=\nu^{0}$). Since $\Omega$ is strictly convex, there also exists a $\theta>0$ depending on $x^{0}$ and $\Omega$ such that
\begin{equation}\label{3-33}
  I:=\left\{\nu\in\mathbb{R}^n\,|\,|\nu|=1, \, \nu\cdot\nu^{0}\geq\cos\theta\right\}\subset\left\{\nu_{x}\,|\,x\in B_{r_{0}}(x^{0})\cap\partial\Omega\right\},
\end{equation}
and hence, we have, for any $x\in B_{r_{0}}(x^{0})\cap\partial\Omega$ and $\nu\in I$,
\begin{equation}\label{3-34}
  u(x+s\nu) \quad \text{is monotone increasing with respect to} \,\,\, s\in\left[0,\frac{\delta_{0}}{2}\right].
\end{equation}
Let
\begin{equation}\label{3-35}
  D:=\{x+r_{0}\nu^{0}\,|\,x\in B_{r_{0}}(x^{0})\cap\partial\Omega\},
\end{equation}
one can easily verify that
\begin{equation}\label{3-36}
  \max_{\overline{B_{r_{0}}(x^{0})\cap\Omega}}u(x)\leq\max_{\overline{D}}u(x).
\end{equation}
For any $x\in\overline{D}$, let
\begin{equation}\label{3-37}
  \overline{V_{x}}:=\left\{x+\nu\,\Big|\,\nu\cdot\nu^{0}\geq|\nu|\cos\theta, \, |\nu|\leq\frac{\delta_{0}}{4}\right\}
\end{equation}
be a piece of cone with vertex at $x$, then it is easy to see that
\begin{equation}\label{3-38}
  u(x)=\min_{z\in\overline{V_{x}}}u(z).
\end{equation}

Now we need the following Lemma to control the integral of $u$ on $\overline{V_{x}}$.
\begin{lem}\label{lemma2}
Let $\lambda_{1}$ be the first eigenvalue for $(-\Delta)^{m}$ in $\Omega$ with Navier boundary condition, and $0<\phi\in C^{2m}(\Omega)\cap C^{2m-2}(\overline{\Omega})$ be the corresponding eigenfunction (without loss of generality, we may assume $\|\phi\|_{L^{\infty}(\overline{\Omega})}=1$), i.e.,
\begin{equation*}\\\begin{cases}
(-\Delta)^{ m}\phi(x)=\lambda_{1}\phi(x) \,\,\,\,\,\,\,\,\,\, \text{in} \,\,\, \Omega, \\
\phi(x)=-\Delta \phi(x)=\cdots=(-\Delta)^{ m-1}\phi(x)=0 \,\,\,\,\,\,\,\, \text{on} \,\,\, \partial\Omega.
\end{cases}\end{equation*}
Then, we have
\begin{equation*}
  \int_{\Omega}u^{p}(x)\phi(x)dx\leq C(\lambda_{1},p,|\Omega|).
\end{equation*}
\end{lem}
\begin{proof}
Multiply both side of \eqref{tNavier} by the eigenfunction $\phi(x)$ and integrate by parts, one gets
\begin{eqnarray}\label{3-39}
  \int_{\Omega}u^{p}(x)\phi(x)dx&\leq&\int_{\Omega}(u^{p}(x)+t)\phi(x)dx=\int_{\Omega}(-\Delta)^{ m}u(x)\cdot\phi(x)dx \\
 \nonumber &=&\int_{\Omega}u(x)\cdot(-\Delta)^{ m}\phi(x)dx=\lambda_{1}\int_{\Omega}u(x)\phi(x)dx.
\end{eqnarray}
By H\"{o}lder's inequality, we have
\begin{equation}\label{3-40}
  \int_{\Omega}u^p(x)\phi(x)dx\leq\lambda_{1}\left(\int_{\Omega}u^{p}(x)\phi(x)dx\right)^{\frac{1}{p}}\left(\int_{\Omega}\phi(x)dx\right)^{\frac{1}{p'}},
\end{equation}
and hence
\begin{equation}\label{3-41}
  \int_{\Omega}u^{p}(x)\phi(x)dx\leq\lambda^{p'}_{1}\int_{\Omega}\phi(x)dx\leq\lambda^{p'}_{1}|\Omega|.
\end{equation}
This completes the proof of Lemma \ref{lemma2}.
\end{proof}

By \eqref{3-38} and Lemma \ref{lemma2}, we see that, for any $x\in\overline{D}$,
\begin{eqnarray}\label{3-42}
  C(\lambda_{1},p,\Omega) &\geq& \int_{\Omega}u^{p}(x)\phi(x)dx\geq\int_{\overline{V_{x}}}u^{p}(z)\phi(z)dz \\
 \nonumber &\geq& u^p(x)|\overline{V_{x}}|\cdot\min_{\overline{\Omega^{r_{0}}}}\phi=:u^{p}(x)\cdot C(n,m,x^{0},\Omega),
\end{eqnarray}
where $\overline{\Omega^{r_{0}}}:=\{x\in\Omega\,|\,dist(x,\partial\Omega)\geq r_{0}\}$, and hence
\begin{equation}\label{3-43}
  u(x)\leq C(n,m,p,x^{0},\lambda_{1},\Omega), \quad\quad \forall x\in\overline{D}.
\end{equation}
Therefore, we arrive at
\begin{equation}\label{3-44}
  \max_{\overline{B_{r_{0}}(x^{0})\cap\Omega}}u(x)\leq\max_{\overline{D}}u(x)\leq C(n,m,p,x^{0},\lambda_{1},\Omega).
\end{equation}

Since $x^{0}\in\partial\Omega$ is arbitrary and $\partial\Omega$ is compact, we can cover $\partial\Omega$ by finite balls $\{B_{r_{k}}(x^{k})\}_{k=0}^{K}$ with centers $\{x^{k}\}_{k=0}^{K}\subset\partial\Omega$ ($K$ depends only on $\Omega$). Therefore, there exists a $\bar{\delta}>0$ depending only on $\Omega$ such that
\begin{equation}\label{3-45}
  \|u\|_{L^{\infty}(\overline{\Omega}_{\bar{\delta}})}\leq\max_{0\leq k\leq K}\max_{\overline{B_{r_{k}}(x^{k})\cap\Omega}}u(x)\leq\max_{0\leq k\leq K}C(n,m,p,x^{k},\lambda_{1},\Omega)=:C(n,m,p,\lambda_{1},\Omega),
\end{equation}
where the boundary layer $\overline{\Omega}_{\bar{\delta}}:=\{x\in\overline{\Omega}\,|\,dist(x,\partial\Omega)\leq\bar{\delta}\}$. This completes the proof of boundary layer estimates under assumption i).

\emph{Case ii)} $1<p\leq\frac{n+2}{n-2}$. Under this assumption, we do not require the convexity of $\Omega$ anymore. Since $\partial\Omega$ is $C^{2m-2}$, there exists a $R_{0}>0$ depending only on $\Omega$ such that, for any $x^{0}\in\partial\Omega$, there exists a $\overline{x^{0}}$ satisfying $\overline{B_{R_{0}}(\overline{x^{0}})}\cap\overline{\Omega}=\{x^{0}\}$. For any $x^{0}\in\partial\Omega$, we define the Kelvin transform centered at $\overline{x^{0}}$ by
\begin{equation}\label{3-46}
  x\mapsto x^{\ast}:=\frac{x-\overline{x^{0}}}{|x-\overline{x^{0}}|^{2}}+\overline{x^{0}}, \quad\quad \Omega\rightarrow\Omega^{\ast}\subset B_{\frac{1}{R_{0}}}(\overline{x^{0}}),
\end{equation}
and hence there exists a small $0<\varepsilon_{0}<\frac{1}{100R_{0}}$ depending on $x^{0}$ and $\Omega$ such that $B_{\varepsilon_{0}}\big((\overline{x^{0}})^{\ast}\big)\cap\partial\Omega^{\ast}$ is strictly convex.

Now we define
\begin{equation}\label{3-47}
  \overline{u}(x^{\ast}):=\frac{1}{|x^{\ast}-\overline{x^{0}}|^{n-2}}u\left(\frac{x^{\ast}-\overline{x^{0}}}{|x^{\ast}-\overline{x^{0}}|^{2}}+\overline{x^{0}}\right),
\end{equation}
\begin{equation}\label{3-48}
  \overline{u_{i}}(x^{\ast}):=\frac{1}{|x^{\ast}-\overline{x^{0}}|^{n-2}}u_{i}\left(\frac{x^{\ast}-\overline{x^{0}}}{|x^{\ast}-\overline{x^{0}}|^{2}}+\overline{x^{0}}\right)
\end{equation}
for $i=1,\cdots, m-1$. Then, we have
\begin{equation}\label{3-48'}
  \overline{u}(x^{\ast})>0, \quad\quad \overline{u_{i}}(x^{\ast})>0 \quad\quad \text{in} \,\,\, \Omega^{\ast},
\end{equation}
and from \eqref{tNavier}, we infer that $\overline{u}(x^{\ast})$ and $\overline{u_{i}}(x^{\ast})$ satisfy
\begin{equation}\label{3-49}
\left\{{\begin{array}{l} {-\Delta \overline{u_{ m-1}}(x^{\ast})=\frac{1}{|x^{\ast}-\overline{x^{0}}|^{\tau}}\overline{u}^{p}(x^{\ast})+\frac{t}{|x^{\ast}-\overline{x^{0}}|^{n+2}}, \,\,\,\,\,\, x^{\ast}\in\Omega^{\ast},}\\  {} \\ {-\Delta \overline{u_{ m-2}}(x^{\ast})=\frac{1}{|x^{\ast}-\overline{x^{0}}|^{4}}\overline{u_{ m-1}}(x^{\ast}), \,\,\,\,\,\, x^{\ast}\in\Omega^{\ast},} \\ \cdots\cdots \\ {-\Delta \overline{u}(x^{\ast})=\frac{1}{|x^{\ast}-\overline{x^{0}}|^{4}}\overline{u_{1}}(x^{\ast}), \,\,\,\,\,\, x^{\ast}\in\Omega^{\ast},} \\ {} \\
{\overline{u}(x^{\ast})=\overline{u_{1}}(x^{\ast})=\cdots=\overline{u_{ m-1}}(x^{\ast})=0, \,\,\,\,\,\, x^{\ast}\in\partial\Omega^{\ast},} \\ \end{array}}\right.
\end{equation}
where $\tau:=n+2-p(n-2)\geq0$. Let $\nu^{0}$ be the unit internal normal vector of $\partial\Omega^{\ast}$ at $(x^{0})^{\ast}$, we will show that $\overline{u}(x^{\ast})$ is monotone increasing along the internal normal direction in the region
\begin{equation}\label{3-50}
  \overline{\Sigma_{\delta_{\ast}}}=\left\{x^{\ast}\in\overline{\Omega^{\ast}}\,|\,0\leq(x^{\ast}-(x^{0})^{\ast})\cdot\nu^{0}\leq\delta_{\ast}\right\},
\end{equation}
where $\delta_{\ast}>0$ depends only on $x^{0}$ and $\Omega$.

For this purpose, we define the moving plane by
\begin{equation}\label{3-51}
  T^{\ast}_{\lambda}:=\{x^{\ast}\in\mathbb{R}^n\,|\,(x^{\ast}-(x^{0})^{\ast})\cdot\nu^{0}=\lambda\},
\end{equation}
and denote
\begin{equation}\label{3-52}
  \Sigma^{\ast}_{\lambda}:=\{x^{\ast}\in\Omega^{\ast}\,|\,0<(x^{\ast}-(x^{0})^{\ast})\cdot\nu^{0}<\lambda\}
\end{equation}
for $\lambda>0$, and let $x^{\ast}_{\lambda}$ be the reflection of the point $x^{\ast}$ about the plane $T^{\ast}_{\lambda}$.

Define
\begin{equation}\label{3-53}
  U^{\lambda}(x^{\ast}):=\overline{u}(x^{\ast}_{\lambda})-\overline{u}(x^{\ast}) \quad\quad\, \text{and} \quad\quad\, U^{\lambda}_{i}(x^{\ast}):=\overline{u_{i}}(x^{\ast}_{\lambda})-\overline{u_{i}}(x^{\ast})
\end{equation}
for $1\leq i\leq m-1$. Then we can deduce from \eqref{3-49} that, for any $\lambda$ satisfying the reflection of $\Sigma^{\ast}_{\lambda}$ is contained in $\Omega^{\ast}$,
\begin{equation}\label{3-54}
\left\{{\begin{array}{l} {-\Delta U^{\lambda}_{ m-1}(x^{\ast})=\frac{\overline{u}^{p}(x^{\ast}_{\lambda})}{|x^{\ast}_{\lambda}-\overline{x^{0}}|^{\tau}}
-\frac{\overline{u}^{p}(x^{\ast})}{|x^{\ast}-\overline{x^{0}}|^{\tau}}
+\frac{t}{|x^{\ast}_{\lambda}-\overline{x^{0}}|^{n+2}}-\frac{t}{|x^{\ast}-\overline{x^{0}}|^{n+2}}, \,\,\,\,\,\, x^{\ast}\in\Sigma^{\ast}_{\lambda},}\\  {} \\ {-\Delta U^{\lambda}_{ m-2}(x^{\ast})=\frac{\overline{u_{ m-1}}(x^{\ast}_{\lambda})}{|x^{\ast}_{\lambda}-\overline{x^{0}}|^{4}}
-\frac{\overline{u_{ m-1}}(x^{\ast})}{|x^{\ast}-\overline{x^{0}}|^{4}}, \,\,\,\,\,\, x^{\ast}\in\Sigma^{\ast}_{\lambda},} \\ \cdots\cdots \\ {-\Delta U^{\lambda}(x^{\ast})=\frac{\overline{u_{1}}(x^{\ast}_{\lambda})}{|x^{\ast}_{\lambda}-\overline{x^{0}}|^{4}}
-\frac{\overline{u_{1}}(x^{\ast})}{|x^{\ast}-\overline{x^{0}}|^{4}}, \,\,\,\,\,\, x^{\ast}\in\Sigma^{\ast}_{\lambda},} \\ {} \\
{U^{\lambda}(x^{\ast})\geq0, \, U^{\lambda}_{1}(x^{\ast})\geq0, \cdots, U^{\lambda}_{ m-1}(x^{\ast})\geq0, \,\,\,\,\,\, x^{\ast}\in\partial\Sigma^{\ast}_{\lambda}.} \\ \end{array}}\right.
\end{equation}
Notice that for any $x^{\ast}\in\Sigma^{\ast}_{\lambda}$ with $\lambda<\frac{1}{R_{0}}$, one has
\begin{equation}\label{3-55}
  0<|x^{\ast}_{\lambda}-\overline{x^{0}}|<|x^{\ast}-\overline{x^{0}}|<\frac{1}{R_{0}},
\end{equation}
and hence, by direct calculations, it follows from \eqref{3-54} and $t\geq0$ that
\begin{equation}\label{3-56}
\left\{{\begin{array}{l} {-\Delta U^{\lambda}_{ m-1}(x^{\ast})\geq\frac{p\xi^{p-1}_{\lambda}(x^{\ast})}{|x^{\ast}-\overline{x^{0}}|^{\tau}}U^{\lambda}(x^{\ast})\geq pR^{\tau}_{0}\xi^{p-1}_{\lambda}(x^{\ast})U^{\lambda}(x^{\ast}), \,\,\,\,\,\, x^{\ast}\in\Sigma^{\ast}_{\lambda},}\\  {} \\ {-\Delta U^{\lambda}_{ m-2}(x^{\ast})\geq R^{4}_{0}\,U^{\lambda}_{ m-1}(x^{\ast}), \,\,\,\,\,\, x^{\ast}\in\Sigma^{\ast}_{\lambda},} \\ \cdots\cdots \\ {-\Delta U^{\lambda}(x^{\ast})\geq R^{4}_{0}\,U^{\lambda}_{1}(x^{\ast}), \,\,\,\,\,\, x^{\ast}\in\Sigma^{\ast}_{\lambda},} \\ {} \\
{U^{\lambda}(x^{\ast})\geq0, \, U^{\lambda}_{1}(x^{\ast})\geq0, \cdots, U^{\lambda}_{ m-1}(x^{\ast})\geq0, \,\,\,\,\,\, x^{\ast}\in\partial\Sigma^{\ast}_{\lambda}.} \\ \end{array}}\right.
\end{equation}
where $\xi_{\lambda}(x^{\ast})$ is valued between $\overline{u}(x^{\ast}_{\lambda})$ and $\overline{u}(x^{\ast})$ by mean value theorem, and thus
\begin{equation}\label{3-57}
  \|\xi_{\lambda}\|_{L^{\infty}(\overline{\Sigma^{\ast}_{\lambda}})}\leq\left(diam\,\Omega+R_{0}\right)^{n-2}\|u\|_{L^{\infty}(\overline{\Omega})}.
\end{equation}
Now, we will prove that there exists some $\delta>0$ sufficiently small (depending on $m$, $p$, $\|u\|_{L^{\infty}(\overline{\Omega})}$ and $\Omega$), such that
\begin{equation}\label{3-58}
  U^{\lambda}(x^{\ast})\geq 0 \quad\quad\, \text{in} \,\, \Sigma^{\ast}_{\lambda}
\end{equation}
for all $0<\lambda\leq\delta$. This provides a starting point to move the plane $T^{\ast}_{\lambda}$.

In fact, suppose on the contrary that there exists a $0<\lambda\leq\delta$ such that
\begin{equation}\label{3-59}
  U^{\lambda}(x^{\ast})<0 \quad\quad\, \text{somewhere in} \,\, \Sigma^{\ast}_{\lambda}.
\end{equation}
Let
\begin{equation}\label{3-60}
  \psi(x^{\ast}):=\cos\frac{(x^{\ast}-(x^{0})^{\ast})\cdot\nu^{0}}{\delta},
\end{equation}
then $\psi(x^{\ast})\in[\cos1,1]$ for any $x^{\ast}\in\Sigma^{\ast}_{\lambda}$ and $-\frac{\Delta\psi}{\psi}=\frac{1}{\delta^2}$. Define
\begin{equation}\label{3-61}
  \overline{U^{\lambda}}(x^{\ast}):=\frac{U^{\lambda}(x^{\ast})}{\psi(x^{\ast})} \quad\quad \text{and} \quad\quad \overline{U^{\lambda}_{i}}(x^{\ast}):=\frac{U^{\lambda}_{i}(x^{\ast})}{\psi(x^{\ast})}
\end{equation}
for $i=1,\cdots, m-1$ and $x^{\ast}\in\Sigma^{\ast}_{\lambda}$. Then there exists a $x^{\ast}_{0}\in\Sigma^{\ast}_{\lambda}$ such that
\begin{equation}\label{3-62}
  \overline{U^{\lambda}}(x^{\ast}_{0})=\min_{\overline{\Sigma^{\ast}_{\lambda}}}\overline{U^{\lambda}}(x^{\ast})<0.
\end{equation}
Since
\begin{equation}\label{3-63}
  -\Delta U^{\lambda}(x^{\ast}_{0})=-\Delta\overline{U^{\lambda}}(x^{\ast}_{0})\psi(x^{\ast}_{0})-2\nabla\overline{U^{\lambda}}(x^{\ast}_{0})\cdot\nabla\psi(x^{\ast}_{0})
  -\overline{U^{\lambda}}(x^{\ast}_{0})\Delta\psi(x^{\ast}_{0}),
\end{equation}
one immediately has
\begin{equation}\label{3-64}
  R^{4}_{0}\,U^{\lambda}_{1}(x^{\ast}_{0})\leq-\Delta U^{\lambda}(x^{\ast}_{0})\leq\frac{1}{\delta^2}U^{\lambda}(x^{\ast}_{0})<0.
\end{equation}
Thus there exists a $x^{\ast}_{1}\in\Sigma^{\ast}_{\lambda}$ such that
\begin{equation}\label{3-65}
  \overline{U^{\lambda}_{1}}(x^{\ast}_{1})=\min_{\overline{\Sigma^{\ast}_{\lambda}}}\overline{U^{\lambda}_{1}}(x^{\ast})<0.
\end{equation}
Similarly, it follows that
\begin{equation}\label{3-66}
  R^{4}_{0}\,U^{\lambda}_{2}(x^{\ast}_{1})\leq-\Delta U^{\lambda}_{1}(x^{\ast}_{1})\leq\frac{1}{\delta^2}U^{\lambda}_{1}(x^{\ast}_{1})<0.
\end{equation}
Continuing this way, we get $\{x^{\ast}_{i}\}_{i=1}^{ m-1}\subset\Sigma^{\ast}_{\lambda}$ such that
\begin{equation}\label{3-67}
  \overline{U^{\lambda}_{i}}(x^{\ast}_{i})=\min_{\overline{\Sigma^{\ast}_{\lambda}}}\overline{U^{\lambda}_{i}}(x^{\ast})<0,
\end{equation}
\begin{equation}\label{3-68}
  R^{4}_{0}\,U^{\lambda}_{i+1}(x^{\ast}_{i})\leq-\Delta U^{\lambda}_{i}(x^{\ast}_{i})\leq\frac{1}{\delta^2}U^{\lambda}_{i}(x^{\ast}_{i})<0
\end{equation}
for $i=1,2,\cdots, m-2$, and
\begin{equation}\label{3-69}
  \overline{U^{\lambda}_{ m-1}}(x^{\ast}_{ m-1})=\min_{\overline{\Sigma^{\ast}_{\lambda}}}\overline{U^{\lambda}_{ m-1}}(x^{\ast})<0,
\end{equation}
\begin{equation}\label{3-70}
  pR^{\tau}_{0}\xi^{p-1}_{\lambda}(x^{\ast}_{ m-1})U^{\lambda}(x^{\ast}_{ m-1})\leq-\Delta U^{\lambda}_{ m-1}(x^{\ast}_{ m-1})\leq\frac{1}{\delta^2}U^{\lambda}_{ m-1}(x^{\ast}_{ m-1})<0.
\end{equation}
Therefore, we have
\begin{eqnarray}\label{3-71}
  U^{\lambda}(x^{\ast}_{0}) &\geq& (\delta R^{2}_{0})^{2}U^{\lambda}_{1}(x^{\ast}_{0})\geq (\delta R^{2}_{0})^{2}U^{\lambda}_{1}(x^{\ast}_{1})\frac{\psi(x^{\ast}_{0})}{\psi(x^{\ast}_{1})} \\
 \nonumber &\geq& (\delta R^{2}_{0})^{4}U^{\lambda}_{2}(x^{\ast}_{1})\frac{\psi(x^{\ast}_{0})}{\psi(x^{\ast}_{1})}
 \geq(\delta R^{2}_{0})^{4}U^{\lambda}_{2}(x^{\ast}_{2})\frac{\psi(x^{\ast}_{0})}{\psi(x^{\ast}_{2})} \\
 \nonumber &\geq& (\delta R^{2}_{0})^{6}U^{\lambda}_{3}(x^{\ast}_{2})\frac{\psi(x^{\ast}_{0})}{\psi(x^{\ast}_{2})}
 \geq(\delta R^{2}_{0})^{6}U^{\lambda}_{3}(x^{\ast}_{3})\frac{\psi(x^{\ast}_{0})}{\psi(x^{\ast}_{3})} \\
 \nonumber  &\geq& \cdots\cdots\geq(\delta R^{2}_{0})^{2m-2}U^{\lambda}_{ m-1}(x^{\ast}_{ m-1})\frac{\psi(x^{\ast}_{0})}{\psi(x^{\ast}_{ m-1})} \\
 \nonumber  &\geq& p\delta^{2m}R^{4m-(p-1)(n-2)}_{0}\xi^{p-1}_{\lambda}(x^{\ast}_{ m-1})U^{\lambda}(x^{\ast}_{ m-1})
 \frac{\psi(x^{\ast}_{0})}{\psi(x^{\ast}_{ m-1})} \\
 \nonumber  &\geq& p\delta^{2m}R^{4m-(p-1)(n-2)}_{0}\left(diam\,\Omega+R_{0}\right)^{(p-1)(n-2)}\|u\|^{p-1}_{L^{\infty}(\overline{\Omega})}U^{\lambda}(x^{\ast}_{0}),
\end{eqnarray}
that means,
\begin{equation}\label{3-72}
  1\leq p\delta^{2m}\left(diam\,\Omega+R_{0}\right)^{4m}\|u\|^{p-1}_{L^{\infty}(\overline{\Omega})},
\end{equation}
which is absurd if we choose $\delta>0$ small enough such that
\begin{equation}\label{3-73}
  0<\delta<\left(diam\,\Omega+R_{0}\right)^{-2}\left(p\|u\|^{p-1}_{L^{\infty}(\overline{\Omega})}\right)^{-\frac{1}{2m}}.
\end{equation}
So far, we have proved that the plane $T^{\ast}_{\lambda}$ can be moved on up to $\lambda=\delta$.

Next, we will move the plane $T^{\ast}_{\lambda}$ further along the internal normal direction at $(x^{0})^{\ast}$ as long as the property
\begin{equation}\label{3-74}
  U^{\lambda}(x^{\ast})\geq0 \quad\quad\, \text{in} \,\, \Sigma^{\ast}_{\lambda}
\end{equation}
holds. Completely similar to the proof of \emph{Case i)}, one can actually show that the method of moving planes can be carried on (with the property \eqref{3-74}) as long as the reflection of $\overline{\Sigma^{\ast}_{\lambda}}$ is still contained in $\Omega^{\ast}$. We omit the details here.

Therefore, there exists a $\delta_{\ast}>0$ depending only on $x^{0}$ and $\Omega$ such that, $\overline{u}(x^{\ast})$ is monotone increasing along the internal normal direction in the region
\begin{equation}\label{3-75}
  \overline{\Sigma_{\delta_{\ast}}}:=\left\{x^{\ast}\in\overline{\Omega^{\ast}}\,|\,0\leq\left(x^{\ast}-(x^{0})^{\ast}\right)\cdot\nu^{0}\leq\delta_{\ast}\right\}.
\end{equation}
Since $\partial\Omega^{\ast}$ is $C^{2m-2}$, there exists a small $0<\varepsilon_{1}<\min\{\frac{\delta_{\ast}}{8},\varepsilon_{0}\}$ depending on $x^{0}$ and $\Omega$ such that, for any $x^{\ast}\in B_{\varepsilon_{1}}\big((x^{0})^{\ast}\big)\cap\partial\Omega^{\ast}$, $\overline{u}(x^{\ast})$ is monotone increasing along the internal normal direction at $x^{\ast}$ in the region
\begin{equation}\label{3-76}
  \overline{\Sigma_{x^{\ast}}}:=\left\{z^{\ast}\in\overline{\Omega^{\ast}}\,\Big|\,0\leq(z^{\ast}-x^{\ast})\cdot\nu_{x^{\ast}}\leq\frac{3}{4}\delta_{\ast}\right\}.
\end{equation}
where $\nu_{x^{\ast}}$ denotes the unit internal normal vector at the point $x^{\ast}$ ($\nu_{(x^{0})^{\ast}}:=\nu^{0}$). Since $B_{\varepsilon_{1}}\big((x^{0})^{\ast}\big)\cap\partial\Omega^{\ast}$ is strictly convex, there exists a $\theta>0$ depending on $x^{0}$ and $\Omega$ such that
\begin{equation}\label{3-77}
  S:=\left\{\nu^{\ast}\in\mathbb{R}^n\,|\,|\nu^{\ast}|=1, \, \nu^{\ast}\cdot\nu^{0}\geq\cos\theta\right\}\subset\left\{\nu_{x^{\ast}}\,|\,x^{\ast}\in B_{\varepsilon_{1}}\big((x^{0})^{\ast}\big)\cap\partial\Omega^{\ast}\right\},
\end{equation}
and hence, it follows that, for any $x^{\ast}\in B_{\varepsilon_{1}}\big((x^{0})^{\ast}\big)\cap\partial\Omega^{\ast}$ and $\nu^{\ast}\in S$,
\begin{equation}\label{3-78}
  \overline{u}(x^{\ast}+s\nu^{\ast}) \quad \text{is monotone increasing with respect to} \,\,\, s\in\left[0,\frac{\delta_{\ast}}{2}\right].
\end{equation}
Now, let
\begin{equation}\label{3-79}
  D^{\ast}:=\left\{x^{\ast}+\varepsilon_{1}\nu^{0}\,|\,x^{\ast}\in B_{\varepsilon_{1}}\big((x^{0})^{\ast}\big)\cap\partial\Omega^{\ast}\right\},
\end{equation}
one immediately has
\begin{equation}\label{3-80}
  \max_{\overline{B_{\varepsilon_{1}}((x^{0})^{\ast})\cap\Omega^{\ast}}}\overline{u}(x^{\ast})\leq\max_{\overline{D^{\ast}}}\overline{u}(x^{\ast}).
\end{equation}
For any $x^{\ast}\in\overline{D^{\ast}}$, let
\begin{equation}\label{3-81}
  \overline{V_{x^{\ast}}}:=\left\{x^{\ast}+\nu^{\ast}\,\Big|\,\nu^{\ast}\cdot\nu^{0}\geq|\nu^{\ast}|\cos\theta, \, |\nu^{\ast}|\leq\frac{\delta_{\ast}}{4}\right\}
\end{equation}
be a piece of cone with vertex at $x^{\ast}$, then it is obvious that
\begin{equation}\label{3-82}
  \overline{u}(x^{\ast})=\min_{z^{\ast}\in\overline{V_{x^{\ast}}}}\overline{u}(z^{\ast}).
\end{equation}

Therefore, by \eqref{3-82} and Lemma \ref{lemma2}, we get, for any $x^{\ast}\in\overline{D^{\ast}}$,
\begin{eqnarray}\label{3-83}
 && C(\lambda_{1},p,\Omega)\geq\int_{\Omega}u^{p}(x)\phi(x)dx \\
 \nonumber &=&\int_{\Omega^{\ast}}\frac{\overline{u}^{p}(x^{\ast})}{|x^{\ast}-\overline{x^{0}}|^{2n-p(n-2)}}
  \phi\left(\frac{x^{\ast}-\overline{x^{0}}}{|x^{\ast}-\overline{x^{0}}|^2}+\overline{x^{0}}\right)dx^{\ast} \\
 \nonumber &\geq& \int_{\overline{V_{x^{\ast}}}}\frac{\overline{u}^{p}(z^{\ast})}{|z^{\ast}-\overline{x^{0}}|^{2n-p(n-2)}}
  \phi\left(\frac{z^{\ast}-\overline{x^{0}}}{|z^{\ast}-\overline{x^{0}}|^2}+\overline{x^{0}}\right)dz^{\ast}\\
 \nonumber &\geq& \overline{u}^p(x^{\ast})R^{2n-p(n-2)}_{0}|\overline{V_{x^{\ast}}}|\cdot\min_{\overline{\Omega^{r_{1}}}}\phi
 =:\overline{u}^{p}(x^{\ast})\cdot C(n,m,p,x^{0},\Omega),
\end{eqnarray}
where $\overline{\Omega^{r_{1}}}:=\{x\in\Omega\,|\,dist(x,\partial\Omega)\geq r_{1}\}$ with $r_{1}=\varepsilon_{1}R^{2}_{0}$, and hence
\begin{equation}\label{3-84}
  \overline{u}(x^{\ast})\leq C(n,m,p,x^{0},\lambda_{1},\Omega), \quad\quad \forall x\in\overline{D^{\ast}}.
\end{equation}
As a consequence, we derive that
\begin{equation}\label{3-85}
  \max_{\overline{B_{\varepsilon_{1}}((x^{0})^{\ast})\cap\Omega^{\ast}}}\overline{u}(x^{\ast})\leq\max_{\overline{D^{\ast}}}\overline{u}(x^{\ast})\leq C(n,m,p,x^{0},\lambda_{1},\Omega).
\end{equation}
There exists a small $r_{0}>0$ depending only on $x^{0}$ and $\Omega$ such that, for each $x\in\overline{B_{r_{0}}(x^{0})\cap\Omega}$, one has $x^{\ast}\in\overline{B_{\varepsilon_{1}}\big((x^{0})^{\ast}\big)\cap\Omega^{\ast}}$. Therefore, \eqref{3-85} yields
\begin{eqnarray}\label{3-86}
  \max_{\overline{B_{r_{0}}(x^{0})\cap\Omega}}u(x)&=&\max_{x\in\overline{B_{r_{0}}(x^{0})\cap\Omega}}|x^{\ast}-\overline{x^{0}}|^{n-2}\overline{u}(x^{\ast}) \\
  \nonumber &\leq& \frac{1}{R^{n-2}_{0}}\max_{\overline{B_{\varepsilon_{1}}((x^{0})^{\ast})\cap\Omega^{\ast}}}\overline{u}(x^{\ast})\leq C(n,m,p,x^{0},\lambda_{1},\Omega).
\end{eqnarray}

Since $x^{0}\in\partial\Omega$ is arbitrary and $\partial\Omega$ is compact, we can cover $\partial\Omega$ by finite balls $\{B_{r_{k}}(x^{k})\}_{k=0}^{K}$ with centers $\{x^{k}\}_{k=0}^{K}\subset\partial\Omega$ ($K$ depends only on $\Omega$). Therefore, there exists a $\bar{\delta}>0$ depending only on $\Omega$ such that
\begin{equation}\label{3-87}
  \|u\|_{L^{\infty}(\overline{\Omega}_{\bar{\delta}})}\leq\max_{0\leq k\leq K}\max_{\overline{B_{r_{k}}(x^{k})\cap\Omega}}u(x)\leq\max_{0\leq k\leq K}C(n,m,p,x^{k},\lambda_{1},\Omega)=:C(n,m,p,\lambda_{1},\Omega),
\end{equation}
where the boundary layer $\overline{\Omega}_{\bar{\delta}}:=\{x\in\overline{\Omega}\,|\,dist(x,\partial\Omega)\leq\bar{\delta}\}$. This completes the proof of boundary layer estimates under assumption ii).

This concludes our proof of Theorem \ref{Boundary}.
\end{proof}

\subsection{Blowing-up analysis and interior estimates}

In this subsection, we will obtain the interior estimates (and hence, global a priori estimates) via the blowing-up analysis arguments (for related literatures on blowing-up methods, please refer to \cite{BC,BM,CDQ,CL3,CL4,CY0,Li,SZ1}).

Suppose on the contrary that Theorem \ref{Thm1} does not hold. By the boundary layer estimates (Theorem \ref{Boundary}), there exists a sequence of positive solutions $\{u_{k}\}\subset C^{2m}(\Omega)\cap C^{2m-2}(\overline{\Omega})$ to the higher order Navier problem \eqref{tNavier} and a sequence of interior points $\{x^{k}\}\subset\Omega\setminus\overline{\Omega}_{\bar{\delta}}$ such that
\begin{equation}\label{32-1}
  m_{k}:=u_{k}(x^{k})=\|u_{k}\|_{L^{\infty}(\overline{\Omega})}\rightarrow+\infty \quad \text{as} \,\, k\rightarrow\infty.
\end{equation}
For $x\in\Omega_{k}:=\{x\in\mathbb{R}^{n}\,|\,\lambda_{k}x+x^{k}\in\Omega\}$, we define
\begin{equation}\label{32-2}
  v_{k}(x):=\frac{1}{m_{k}}u_{k}(\lambda_{k}x+x^{k}) \quad \text{with} \,\, \lambda_{k}:=m_{k}^{\frac{1-p}{2m}}\rightarrow0 \quad \text{as} \,\, k\rightarrow\infty.
\end{equation}
Then $v_{k}(x)$ satisfies $\|v_{k}\|_{L^{\infty}(\overline{\Omega_{k}})}=v_{k}(0)=1$ and
\begin{eqnarray}\label{32-3}
  (-\Delta)^{m}v_{k}(x)&=&\frac{1}{m_{k}}\lambda^{2m}_{k}(-\Delta)^{m}u_{k}(\lambda_{k}x+x^{k}) \\
 \nonumber &=& \frac{1}{m_{k}}\lambda^{2m}_{k}\left(u^{p}_{k}(\lambda_{k}x+x^{k})+t\right)=v^{p}_{k}(x)+\frac{t}{m^{p}_{k}}
\end{eqnarray}
for any $x\in\Omega_{k}$. Since $dist(x^{k},\partial\Omega)>\bar{\delta}$, one has
\begin{equation}\label{32-4}
  \Omega_{k}\supset\left\{x\in\mathbb{R}^{n}\,|\,|\lambda_{k}x|\leq\bar{\delta}\right\}=\overline{B_{\frac{\bar{\delta}}{\lambda_{k}}}(0)},
\end{equation}
and hence
\begin{equation}\label{32-5}
  \Omega_{k}\rightarrow\mathbb{R}^{n} \quad \text{as} \,\, k\rightarrow\infty.
\end{equation}

For arbitrary $x^{0}\in\mathbb{R}^{n}$, there exists a $N_{1}>0$, such that $\overline{B_{1}(x^{0})}\subset\Omega_{k}$ for any $k\geq N_{1}$. By \eqref{32-3} and $\|v_{k}\|_{L^{\infty}(\overline{\Omega_{k}})}\leq1$, we can infer from regularity theory and Sobolev embedding that
\begin{equation}\label{32-6}
  \|v_{k}\|_{C^{2m-1,\gamma}(\overline{B_{1}(0)})}\leq C(1+t),
\end{equation}
and further that
\begin{equation}\label{32-7}
   \|v_{k}\|_{C^{2(2m-1),\gamma}(\overline{B_{1}(0)})}\leq C(1+t)
\end{equation}
for $k\geq N_{1}$, where $0\leq\gamma<1$. As a consequence, by Arzel\`{a}-Ascoli Theorem, there exists a subsequence $\{v^{(1)}_{k}\}\subset\{v_{k}\}$ and a function $v\in C^{2m}(\overline{B_{1}(x^{0})})$ such that
\begin{equation}\label{32-8}
  v^{(1)}_{k}\rightrightarrows v \quad \text{and} \quad (-\Delta)^{ m}v^{(1)}_{k}\rightrightarrows(-\Delta)^{ m}v \quad\quad \text{in} \,\, \overline{B_{1}(x^{0})}.
\end{equation}
There also exists a $N_{2}>0$ such that $\overline{B_{2}(x^{0})}\subset\Omega_{k}$ for any $k\geq N_{2}$. By \eqref{32-3} and $\|v_{k}\|_{L^{\infty}(\overline{\Omega_{k}})}\leq1$, we can deduce that
\begin{equation}\label{32-9}
  \|v^{(1)}_{k}\|_{C^{2(2m-1),\gamma}(\overline{B_{2}(0)})}\leq C(1+t)
\end{equation}
for $k\geq N_{2}$, where $0\leq\gamma<1$. Therefore, by Arzel\`{a}-Ascoli Theorem again, there exists a subsequence $\{v^{(2)}_{k}\}\subset\{v^{(1)}_{k}\}$ and $v\in C^{2m}(\overline{B_{2}(x^{0})})$ such that
\begin{equation}\label{32-10}
  v^{(2)}_{k}\rightrightarrows v \quad \text{and} \quad (-\Delta)^{ m}v^{(2)}_{k}\rightrightarrows(-\Delta)^{ m}v \quad\quad \text{in} \,\, \overline{B_{2}(x^{0})}.
\end{equation}
Continuing this way, for any $j\in\mathbb{N}^{+}$, we can extract a subsequence $\{v^{(j)}_{k}\}\subset\{v^{(j-1)}_{k}\}$ and find a function $v\in C^{2m}(\overline{B_{j}(x^{0})})$ such that
\begin{equation}\label{32-11}
  v^{(j)}_{k}\rightrightarrows v \quad \text{and} \quad (-\Delta)^{ m}v^{(j)}_{k}\rightrightarrows(-\Delta)^{ m}v \quad\quad \text{in} \,\, \overline{B_{j}(x^{0})}.
\end{equation}
By extracting the diagonal sequence, we finally obtain that the subsequence $\{v^{(k)}_{k}\}$ satisfies
\begin{equation}\label{32-12}
  v^{(k)}_{k}\rightrightarrows v \quad \text{and} \quad (-\Delta)^{ m}v^{(k)}_{k}\rightrightarrows(-\Delta)^{ m}v \quad\quad \text{in} \,\, \overline{B_{j}(x^{0})}
\end{equation}
for any $j\geq1$. Therefore, we get from \eqref{32-3} that $0\leq v\in C^{2m}(\mathbb{R}^{n})$ satisfies
\begin{equation}\label{32-13}
  (-\Delta)^{ m}v(x)=v^{p}(x) \quad\quad \text{in} \,\, \mathbb{R}^{n}.
\end{equation}
By the Liouville theorem (Theorem \ref{Thm0}), we must have $v\equiv0$ in $\mathbb{R}^{n}$, which is a contradiction with
\begin{equation}\label{32-14}
  v(0)=\lim_{k\rightarrow\infty}v^{(k)}_{k}(0)=1.
\end{equation}

This concludes our proof of Theorem \ref{Thm1}.

\section{Proof of Theorem \ref{Thm2}}

In this section, by applying the a priori estimates (Theorem \ref{Thm-CFL} and Theorem \ref{Thm1}) and the following Leray-Schauder fixed point theorem (see e.g. \cite{CLM,CDQ}), we will prove the existence of positive solutions to the higher order Lane-Emden equations \eqref{Navier} with Navier boundary conditions.
\begin{thm}\label{L-S}
Suppose that $X$ is a real Banach space with a closed positive cone $P$, $U\subset P$ is bounded open and contains $0$. Assume that there exists $\rho>0$ such that $B_{\rho}(0)\cap P\subset U$ and that $T:\,\overline{U}\rightarrow P$ is compact and satisfies
\vskip 5pt
\noindent i) For any $x\in P$ with $|x|=\rho$ and any $\lambda\in[0,1)$, $x\neq\lambda Tx$;
\vskip 3pt
\noindent ii) There exists some $y\in P\setminus\{0\}$ such that $x-Tx\neq ty$ for any $t\geq0$ and $x\in\partial U$.
\vskip 5pt
\noindent Then, $T$ possesses a fixed point in $\overline{U_{\rho}}$, where $U_{\rho}:=U\setminus B_{\rho}(0)$.
\end{thm}

Now we let
\begin{equation}\label{4-1}
  X:=C^{0}(\overline{\Omega}) \quad\quad \text{and} \quad\quad P:=\{u\in X \,|\, u\geq0\}.
\end{equation}
Define
\begin{equation}\label{4-2}
  T(u)(x):=\int_{\Omega}G_{2}(x,y^{ m})\int_{\Omega}G_{2}(y^{ m},y^{ m-1})\int_{\Omega}\cdots\int_{\Omega}G_{2}(y^{2},y^{1})u^{p}(y^{1})
  dy^{1}dy^{2}\cdots dy^{ m},
\end{equation}
where $G_{2}(x,y)$ is the Green's function for $-\Delta$ with Dirichlet boundary condition in $\Omega$. Suppose $u\in C^{0}(\overline{\Omega})$ is a fixed point of $T$, i.e., $u=Tu$, then it is easy to see that $u\in C^{2m}(\Omega)\cap C^{2m-2}(\overline{\Omega})$ and satisfies the Navier problem
\begin{equation}\label{4-3}\\\begin{cases}
(-\Delta)^{ m}u(x)=u^{p}(x) \,\,\,\,\,\,\,\,\,\, \text{in} \,\,\, \Omega, \\
u(x)=-\Delta u(x)=\cdots=(-\Delta)^{ m-1}u(x)=0 \,\,\,\,\,\,\,\, \text{on} \,\,\, \partial\Omega.
\end{cases}\end{equation}

Our goal is to show the existence of a fixed point for $T$ in $P\setminus B_{\rho}(0)$ for some $\rho>0$ (to be determined later) by using Theorem \ref{L-S}. To this end, we need to verify the two conditions i) and ii) in Theorem \ref{L-S} separately.

\emph{i)} First, we show that there exists $\rho>0$ such that for any $u\in\partial B_{\rho}(0)\cap P$ and $0\leq\lambda<1$,
\begin{equation}\label{4-4}
  u-\lambda T(u)\neq0.
\end{equation}
For any $x\in\overline{\Omega}$, it holds that
\begin{eqnarray}\label{4-5}
 \nonumber |T(u)(x)|&=&\left|\int_{\Omega}G_{2}(x,y^{ m})\int_{\Omega}G_{2}(y^{ m},y^{ m-1})\cdots\int_{\Omega}G_{2}(y^{2},y^{1})u^{p}(y^{1})
  dy^{1}\cdots dy^{ m}\right| \\
  &\leq& \int_{\Omega}G_{2}(x,y^{ m})\int_{\Omega}G_{2}(y^{ m},y^{ m-1})\cdots\int_{\Omega}G_{2}(y^{2},y^{1})dy^{1}\cdots dy^{ m}\cdot\|u\|^{p}_{C^{0}(\overline{\Omega})} \\
 \nonumber &\leq& \rho^{p-1}\left\|\int_{\Omega}G_{2}(x,y)dy\right\|^{ m}_{C^{0}(\overline{\Omega})}\cdot\|u\|_{C^{0}(\overline{\Omega})}.
\end{eqnarray}
Let $g(x):=\int_{\Omega}G_{2}(x,y)dy$, then it solves
\begin{equation}\label{4-6}\\\begin{cases}
-\Delta_{x}g(x)=1 \,\,\,\,\,\,\,\,\,\, \text{in} \,\,\, \Omega, \\
g(x)=0 \,\,\,\,\,\,\,\, \text{on} \,\,\, \partial\Omega.
\end{cases}\end{equation}
For a fixed point $x^{0}\in\Omega$, we define the function
\begin{equation}\label{4-7}
  \beta(x):=\frac{(diam\,\Omega)^{2}}{2n}\left(1-\frac{|x-x^{0}|^{2}}{(diam\,\Omega)^{2}}\right)_{+},
\end{equation}
then it satisfies
\begin{equation}\label{4-8}\\\begin{cases}
-\Delta_{x}\beta(x)=1 \,\,\,\,\,\,\,\,\,\, \text{in} \,\,\, \Omega, \\
\beta(x)>0 \,\,\,\,\,\,\,\, \text{on} \,\,\, \partial\Omega.
\end{cases}\end{equation}
By maximum principle, we get
\begin{equation}\label{4-9}
  0\leq g(x)<\beta(x)\leq\frac{(diam\,\Omega)^{2}}{2n}, \quad\quad \forall \,\, x\in\overline{\Omega}.
\end{equation}
Therefore, we infer from \eqref{4-5} and \eqref{4-9} that
\begin{equation}\label{4-10}
  \|T(u)\|_{C^{0}(\overline{\Omega})}<\rho^{p-1}\frac{(diam\,\Omega)^{2m}}{(2n)^{m}}\|u\|_{C^{0}(\overline{\Omega})}=\|u\|_{C^{0}(\overline{\Omega})}
\end{equation}
if we take
\begin{equation}\label{4-11}
  \rho=\left(\frac{\sqrt{2n}}{diam\,\Omega}\right)^{\frac{2m}{p-1}}>0.
\end{equation}
This implies that $u\neq\lambda T(u)$ for any $u\in\partial B_{\rho}(0)\cap P$ and $0\leq\lambda<1$.

\emph{ii)} Now, let $\varphi\in C^{2m}(\Omega)\cap C^{2m-2}(\overline{\Omega})$ be the unique positive solution of
\begin{equation}\label{4-12}\\\begin{cases}
(-\Delta)^{ m}\varphi(x)=1, \,\,\,\,\,\,\,\,\,\, \,\,\, x\in\Omega, \\
\varphi(x)=-\Delta\varphi(x)=\cdots=(-\Delta)^{m-1}\varphi(x)=0, \,\,\,\,\,\,\,\,\,\,\,\, x\in\partial\Omega.
\end{cases}\end{equation}
We will show that
\begin{equation}\label{4-13}
  u-T(u)\neq t\varphi \quad\quad \forall \,\, t\geq0, \quad \forall u\in\partial U,
\end{equation}
where $U:=B_{R}(0)\cap P$ with sufficiently large $R>\rho$ (to be determined later). First, observe that for any $u\in\overline{U}$,
\begin{equation}\label{4-14}
  \left\|(-\Delta)^{ m}T(u)\right\|_{C^{0}(\overline{\Omega})}=\|u\|^{p}_{C^{0}(\overline{\Omega})}\leq R^{p},
\end{equation}
and hence
\begin{equation}\label{4-15}
  \|T(u)\|_{C^{0,\alpha}(\Omega)}\leq CR^{p} \quad\quad \forall \,\, 0<\alpha<1,
\end{equation}
thus $T:\, \overline{U}\rightarrow P$ is compact.

We use contradiction arguments to prove \eqref{4-13}. Suppose on the contrary that, there exists some $u\in\partial U$ and $t\geq0$ such that
\begin{equation}\label{4-16}
  u-T(u)=t\varphi,
\end{equation}
then one has $\|u\|_{C^{0}(\overline{\Omega})}=R>\rho>0$, $u\in C^{2m}(\Omega)\cap C^{2m-2}(\overline{\Omega})$ and satisfies the Navier problem
\begin{equation}\label{4-17}\\\begin{cases}
(-\Delta)^{ m}u(x)=u^{p}(x)+t, \,\,\,\,\,\,\,\,\, u(x)>0, \,\,\,\,\,\,\,\,\,\, x\in\Omega, \\
u(x)=-\Delta u(x)=\cdots=(-\Delta)^{ m-1}u(x)=0, \,\,\,\,\,\,\,\,\,\,\,\, x\in\partial\Omega.
\end{cases}\end{equation}
Choose a constant $C_{1}>\lambda_{1}$. Since $u(x)>0$ in $\Omega$ and $p>1$, it is easy to see that, there exists another constant $C_{2}>0$ (e.g., take $C_{2}=C_{1}^{\frac{p}{p-1}}$), such that
\begin{equation}\label{4-18}
  u^{p}(x)\geq C_{1}u(x)-C_{2}.
\end{equation}
If $t\geq C_{2}$, then we have
\begin{equation}\label{4-19}
  (-\Delta)^{ m}u(x)=u^{p}(x)+t\geq C_{1}u(x)-C_{2}+t\geq C_{1}u(x) \quad\quad \text{in} \,\, \Omega.
\end{equation}
Multiplying both side of \eqref{4-19} by the eigenfunction $\phi(x)$, and integrating by parts yield
\begin{eqnarray}\label{4-20}
  C_{1}\int_{\Omega}u(x)\phi(x)dx&\leq&\int_{\Omega}(-\Delta)^{ m}u(x)\cdot\phi(x)dx=\int_{\Omega}u(x)\cdot(-\Delta)^{ m}\phi(x)dx \\
 \nonumber &=&\lambda_{1}\int_{\Omega}u(x)\phi(x)dx,
\end{eqnarray}
and hence
\begin{equation}\label{4-21}
  0<(C_{1}-\lambda_{1})\int_{\Omega}u(x)\phi(x)dx\leq0,
\end{equation}
which is absurd. Thus, we must have $0\leq t<C_{2}$. Next, we carry on our proof by discussing two different assumptions.

If $\frac{n}{n-2m}<p<\frac{n+2m}{n-2m}$, by the a priori estimates (Theorem \ref{Thm-CFL}), we derive that
\begin{equation}\label{4-CFL}
  \|u\|_{L^{\infty}(\overline{\Omega})}\leq C(n,m,p,\Omega)=:C'_{0}.
\end{equation}

If $\Omega$ is strictly convex, $1<p<\frac{n+2m}{n-2m}$, or if $1<p\leq\frac{n+2}{n-2}$, by the a priori estimates (Theorem \ref{Thm1}), we know that
\begin{equation}\label{4-22}
  \|u\|_{L^{\infty}(\overline{\Omega})}\leq C(n,m,p,t,\lambda_{1},\Omega).
\end{equation}
We will show that the above a priori estimates \eqref{4-22} are uniform with respect to $0\leq t<C_{2}$, i.e., for $0\leq t<C_{2}$,
\begin{equation}\label{4-23}
  \|u\|_{L^{\infty}(\overline{\Omega})}\leq C(n,m,p,C_{2},\lambda_{1},\Omega)=:C''_{0}.
\end{equation}
Indeed, it is clear from Theorem \ref{Boundary} that, the thickness $\bar{\delta}$ of the boundary layer and the boundary layer estimates are uniform with respect to $t$. Therefore, if \eqref{4-23} does not hold, there exist sequences $\{t_{k}\}\subset[0,C_{2})$, $\{x^{k}\}\subset\Omega\setminus\overline{\Omega}_{\bar{\delta}}$ and $\{u_{k}\}$ satisfying
\begin{equation}\label{4-24}\\\begin{cases}
(-\Delta)^{ m}u_{k}(x)=u_{k}^{p}(x)+t_{k}, \,\,\,\,\,\,\,\,\,\, \,\,\, x\in\Omega, \\
u_{k}(x)=-\Delta u_{k}(x)=\cdots=(-\Delta)^{ m-1}u_{k}(x)=0, \,\,\,\,\,\,\,\,\,\,\,\, x\in\partial\Omega,
\end{cases}\end{equation}
but $m_{k}:=u_{k}(x^{k})=\|u_{k}\|_{L^{\infty}(\overline{\Omega})}\rightarrow+\infty$ as $k\rightarrow\infty$. For $x\in\Omega_{k}:=\{x\in\mathbb{R}^{n}\,|\,\lambda_{k}x+x^{k}\in\Omega\}$, we define $v_{k}(x):=\frac{1}{m_{k}}u_{k}(\lambda_{k}x+x^{k})$ with $\lambda_{k}:=m^{\frac{1-p}{2m}}_{k}\rightarrow0$ as $k\rightarrow\infty$. Then $v_{k}(x)$ satisfies $\|v_{k}\|_{L^{\infty}(\overline{\Omega_{k}})}=v_{k}(0)=1$ and
\begin{equation}\label{4-25}
(-\Delta)^{ m}v_{k}(x)=v^{p}_{k}(x)+\frac{t_{k}}{m^{p}_{k}}
\end{equation}
for any $x\in\Omega_{k}$. Since $0\leq t<C_{2}$ and $m_{k}\rightarrow+\infty$, by completely similar blowing-up methods as in the proof of Theorem \ref{Thm1} in subsection 3.2, we can also derive a subsequence $\{v^{(k)}_{k}\}\subset\{v_{k}\}$ and a function $v\in C^{2m}(\mathbb{R}^{n})$ such that
\begin{equation}\label{4-26}
  v^{(k)}_{k}\rightrightarrows v \quad\quad \text{and} \quad\quad (-\Delta)^{m}v^{(k)}_{k}\rightrightarrows(-\Delta)^{m}v \quad\quad \text{in} \,\, \overline{B_{j}(x^{0})}
\end{equation}
for arbitrary $j\geq1$, and hence $0\leq v\in C^{2m}(\mathbb{R}^{n})$ solves
\begin{equation}\label{4-27}
  (-\Delta)^{m}v(x)=v^{p}(x) \quad\quad \text{in} \,\, \mathbb{R}^{n}.
\end{equation}
By Theorem \ref{Thm0}, one immediately has $v\equiv0$, which contradicts with $v(0)=1$. Therefore, the uniform estimates \eqref{4-23} must hold.

Now we let $C_{0}:=\max\{C'_{0},C''_{0}\}>0$ and $R:=C_{0}+\rho$ and $U:=B_{C_{0}+\rho}(0)\cap P$, then \eqref{4-CFL} and \eqref{4-23} implies
\begin{equation}\label{4-28}
  \|u\|_{L^{\infty}(\overline{\Omega})}\leq C_{0}<C_{0}+\rho,
\end{equation}
which contradicts with $u\in\partial U$. This implies that
\begin{equation}\label{4-29}
  u-T(u)\neq t\varphi
\end{equation}
for any $t\geq0$ and $u\in\partial U$ with $U=B_{C_{0}+\rho}(0)\cap P$.

From Theorem \ref{L-S}, we deduce that there exists a $u\in\overline{\big(B_{C_{0}+\rho}(0)\cap P\big)\setminus B_{\rho}(0)}$ satisfies
\begin{equation}\label{4-30}
  u=T(u),
\end{equation}
and hence $\rho\leq\|u\|_{L^{\infty}(\overline{\Omega})}\leq C_{0}+\rho$ solves the higher order Navier problem
\begin{equation}\label{4-31}\\\begin{cases}
(-\Delta)^{ m}u(x)=u^{p}(x), \,\,\,\,\,\,\, u(x)>0, \,\,\,\,\,\,\,\,\, x\in\Omega, \\
u(x)=-\Delta u(x)=\cdots=(-\Delta)^{ m-1}u(x)=0, \,\,\,\,\,\,\,\,\,\,\,\, x\in\partial\Omega.
\end{cases}\end{equation}
By regularity theory, we can see that $u\in C^{2m}(\Omega)\cap C^{2m-2}(\overline{\Omega})$.

This concludes our proof of Theorem \ref{Thm2}.


\begin{thebibliography}{99}

\bibitem{BC}  A. Bahri and J. M. Coron, {\it The scalar-curvature problem on three-dimensional sphere}, J. Funct. Anal., \textbf{95} (1991), 106-172.

\bibitem{BG} M. F. Bidaut-V\'{e}ron and H. Giacomini, {\it A new dynamical approach of Emden-Fowler equations and systems}, Adv. Differential Equations, \textbf{15} (2010), no. 11-12, 1033-1082.

\bibitem{BM} H. Brezis and F. Merle, {\it Uniform estimates and blow-up behavior for solutions of $-\Delta u=V(x)e^{u}$ in two dimensions}, Comm. PDE, \textbf{16} (1991), 1223-1253.

\bibitem{CD} D. Cao and W. Dai, {\it Classification of nonnegative solutions to a bi-harmonic equation with Hartree type nonlinearity}, to appear in Proc. Royal Soc. Edinburgh-A: Math., 2018, 13pp.

\bibitem{CDQ} W. Chen, W. Dai and G. Qin, {\it Liouville type theorems, a priori estimates and existence of solutions for critical order Hardy-H\'{e}non equations in $R^n$}, preprint, submitted, arXiv: 1808.06609.

\bibitem{CF} W. Chen and Y. Fang, {\it Higher order or fractional order Hardy-Sobolev type equations}, Bull. Inst. Math. Acad. Sin. (N.S.), \textbf{9} (2014), no. 3, 317-349.

\bibitem{CFL} W. Chen, Y. Fang and C. Li, {\it Super poly-harmonic property of solutions for Navier boundary problems on a half space}, J. Funct. Anal., \textbf{265} (2013), 1522-1555.

\bibitem{CFY} W. Chen, Y. Fang and R. Yang, {\it Liouville theorems involving the fractional Laplacian on a half space}, Adv. Math., \textbf{274} (2015), 167-198.

\bibitem{CGS} L. Caffarelli, B. Gidas, J. Spruck, {\it Asymptotic symmetry and local behavior of semilinear elliptic equation with critical Sobolev growth}, Comm. Pure Appl. Math., \textbf{42} (1989), 271-297.

\bibitem{CL} W. Chen and C. Li, {\it Classification of solutions of some nonlinear elliptic equations}, Duke Math. J., \textbf{63} (1991), no. 3, 615-622.

\bibitem{CL3} W. Chen and C. Li, {\it A priori estimates for solutions to nonlinear elliptic equations}, Arch. Rational Mech. Anal., \textbf{122} (1993), 145-157.

\bibitem{CL4} W. Chen and C. Li, {\it A priori estimates for prescribing scalar curvature equations}, Annals of Math., \textbf{145} (1997), no. 3, 547-564.

\bibitem{CL1} W. Chen and C. Li, {\it Methods on Nonlinear Elliptic Equations}, AIMS Book Series on Diff. Equa. and Dyn. Sys., Vol. 4, 2010.

\bibitem{CL2} W. Chen and C. Li, {\it Super polyharmonic property of solutions for PDE systems and its applications}, Comm. Pure Appl. Anal., \textbf{12} (2013), 2497-2514.

\bibitem{CLin} C. Chen and C. Lin, {\it Local behavior of singular positive solutions of semilinear elliptic equations with Sobolev exponent}, Duke Math. J., \textbf{78} (1995), no. 2, 315-334.

\bibitem{CLiu} T. Cheng and S. Liu, {\it A Liouville type theorem for higher order Hardy-H\'{e}non equation in $\mathbb{R}^{n}$}, J. Math. Anal. Appl., \textbf{444} (2016), 370-389.

\bibitem{CLL} W. Chen, C. Li and Y. Li, {\it A direct method of moving planes for the fractional Laplacian}, Adv. Math., \textbf{308} (2017), 404-437.

\bibitem{CLM} W. Chen, Y. Li and P. Ma, {\it The Fractional Laplacian}, accepted for publication by World Scientic Publishing Co. Pte. Ltd., 2019, 350pp, https://doi.org/10.1142/10550.

\bibitem{CLO} W. Chen, C. Li and B. Ou, {\it Classification of solutions for an integral equation}, Comm. Pure Appl. Math., \textbf{59} (2006), 330-343.

\bibitem{CY0} S. A. Chang and P. C. Yang, {\it A perturbation result in prescribing scalar curvature on $S^{n}$}, Duke Math. J., \textbf{64} (1991), 27-69.

\bibitem{CY} S. A. Chang and P. C. Yang, {\it On uniqueness of solutions of $n$-th order differential equations in conformal geometry}, Math. Res. Lett., \textbf{4} (1997), 91-102.

\bibitem{CZ} W. Chen and R. Zhang, {\it Classification of solutions and nonlocal curvatures on conformally flat manifolds}, preprint, 2018, 26pp.

\bibitem{DFHQW} W. Dai, Y. Fang, J. Huang, Y. Qin and B. Wang, {\it Regularity and classification of solutions to static Hartree equations involving fractional Laplacians}, Disc. Cont. Dyn. Sys. - A, 2018, 15pp,  doi: 10.3934/dcds.2018117.

\bibitem{DFQ} W. Dai, Y. Fang and G. Qin, {\it Classification of positive solutions to fractional order Hartree equations via a direct method of moving planes}, J. Diff. Equations, \textbf{265} (2018), 2044-2063.

\bibitem{DQ1} W. Dai and G. Qin, {\it Classification of nonnegative classical solutions to third-order equations}, Adv. Math., \textbf{328} (2018), 822-857.

\bibitem{DQ} W. Dai and G. Qin, {\it Liouville type theorems for Hardy-H\'{e}non equations with concave nonlinearities}, preprint, submitted for publication, 2018, 12pp.

\bibitem{FG} M. Fazly and N. Ghoussoub, {\it On the H\'{e}non-Lane-Emden conjecture}, Discrete Contin. Dyn. Syst. - A, \textbf{34} (2014), no. 6, 2513-2533.

\bibitem{GNN1} B. Gidas, W. Ni and L. Nirenberg, {\it Symmetry and related properties via maximum principle}, Comm. Math. Phys., \textbf{68} (1979), 209-243.

\bibitem{GNN} B. Gidas, W. Ni and L. Nirenberg, {\it Symmetry of positive solutions of nonlinear elliptic equations in $\mathbb{R}^{n}$}, Mathematical Analysis and Applications, vol. 7a of the book series Advances in Mathematics, Academic Press, New York, 1981.

\bibitem{GS} B. Gidas and J. Spruck, {\it Global and local behavior of positive solutions of nonlinear elliptic equations}, Comm. Pure Appl. Math., \textbf{34} (1981), no. 4, 525-598.

\bibitem{GS1} B. Gidas and J. Spruck, {\it A priori bounds for positive solutions of nonlinear elliptic equations}, Comm. PDE, \textbf{6} (1981), no. 8, 883-901.

\bibitem{Lei} Y. Lei, {\it Asymptotic properties of positive solutions of the Hardy-Sobolev type equations}, J. Diff. Equations, \textbf{254} (2013), 1774-1799.

\bibitem{Li} Y. Y. Li, {\it Prescribing scalar curvature on $S^{n}$ and related problems, Part I}, J. Diff. Equations, \textbf{120} (1995), 319-410.

\bibitem{LiC} C. Li, {\it Local asymptotic symmetry of singular solutions to nonlinear elliptic equations}, Invent. Math, \textbf{123} (1996), 221-231.

\bibitem{Lin} C. Lin, {\it A classification of solutions of a conformally invariant fourth order equation in $\mathbb{R}^{n}$}, Comment. Math. Helv., \textbf{73} (1998), 206-231.

\bibitem{M} E. Mitidieri, {\it Nonexistence of positive solutions of semilinear elliptic systems in $\mathbb{R}^{N}$}, Differential Integral Equations, \textbf{9} (1996), 465-479.

\bibitem{MP} \`{E}. Mitidieri and S. I. Pohozaev, {\it A priori estimates and the absence of solutions of nonlinear partial differential equations and inequalities}, Tr. Mat. Inst. Steklova, \textbf{234} (2001), 1-384.

\bibitem{P} Q. Phan, {\it Liouville-type theorems for polyharmonic H\'{e}non-Lane-Emden system}, Adv. Nonlinear Stud., \textbf{15} (2015), no. 2, 415-432.

\bibitem{PQS} P. Pol\'{a}\v{c}ik, P. Quittner and P. Souplet, {\it Singularity and decay estimates in superlinear problems via Liouville-type theorems. Part I: Elliptic systems}, Duke Math. J., \textbf{139} (2007), 555-579.

\bibitem{PS} Q. Phan and P. Souplet, {\it Liouville-type theorems and bounds of solutions of Hardy-H\'{e}non equations}, J. Diff. Equations, \textbf{252} (2012), 2544-2562.

\bibitem{S} P. Souplet, {\it The proof of the Lane-Emden conjecture in four space dimensions}, Adv. Math., \textbf{221} (2009), no. 5, 1409-1427.

\bibitem{Stein} E. M. Stein, {\it Singular integrals and differentiability properties of functions}, Princeton Landmarks in Mathematics, Princeton University Press, Princeton, New Jersey, 1970.

\bibitem{SZ} J. Serrin and H. Zou, {\it Non-existence of positive solutions of Lane-Emden systems}, Diff. Integral Equations, \textbf{9} (1996), no. 4, 635-653.

\bibitem{SZ1} R. Schoen and D. Zhang, {\it Prescribed scalar curvature on the $n$-spheres}, Calc. Var. \& PDE, \textbf{4} (1996), no. 1, 1-25.

\bibitem{WX} J. Wei and X. Xu, {\it Classification of solutions of higher order conformally invariant equations}, Math. Ann., \textbf{313} (1999), no. 2, 207-228.

\bibitem{Zhu} N. Zhu, {\it Classification of solutions of a conformally invariant third order equation in $\mathbb{R}^{3}$}, Comm. PDE, \textbf{29} (2004), 1755-1782.

\end{thebibliography}
\end{document}